\documentclass[aos,preprint]{imsart}

\RequirePackage[OT1]{fontenc}
\RequirePackage{amsthm,amsmath}
\RequirePackage[numbers]{natbib}
\startlocaldefs
\theoremstyle{plain}
\newtheorem{thm}{Theorem}
\newtheorem*{thm*}{Theorem}

\newtheorem{lem}{Lemma}
\newtheorem*{lem*}{Lemma}
\newtheorem{defi}{Definition}
\newtheorem*{defi*}{Definition}

\newtheorem*{rques*}{Remarks}
\DeclareMathOperator{\diag}{diag}
\DeclareMathOperator{\sign}{sign}

\endlocaldefs

\usepackage{ulem}
\usepackage{graphicx}
\usepackage{subfigure}

\usepackage{enumitem}

\usepackage{amsfonts,amssymb,amsthm,amsmath}
\usepackage{stmaryrd}
\usepackage{makecell}
\usepackage{capt-of}
\usepackage{dsfont}
\usepackage{tikz}
\usepackage{multirow}
\usepackage{tabularx}
\usepackage{url}

\allowdisplaybreaks

\usepackage{xifthen}
\usepackage{xargs}

\usepackage{mathtools}

\usepackage{longtable}
\usepackage{tabularx}

\usepackage{calrsfs}

\usepackage{array}
\usepackage{makecell}

\usepackage{thmtools,thm-restate}

\begin{document}

\begin{frontmatter}
\title{Improved clustering algorithms for the Bipartite Stochastic Block Model}
\runtitle{ clustering for the Bipartite Stochastic Block Model}
%\thankstext{T1}{Footnote to the title with the ``thankstext'' command.}

\begin{aug}
\author{\fnms{Mohamed} \snm{Ndaoud}\thanksref{m1}\ead[label=e1]{ndaoud@usc.edu}},
\author{\fnms{Suzanne} \snm{Sigalla}\thanksref{m2}\ead[label=e2]{suzanne.sigalla@ensae.fr}}
\and
\author{\fnms{Alexandre B.} \snm{Tsybakov}\thanksref{m2}
\ead[label=e3]{alexandre.tsybakov@ensae.fr}}

%\thankstext{t1}{Some comment}
%\thankstext{t2}{First supporter of the project}
%\thankstext{t3}{Second supporter of the project}
%\runauthor{F. Author et al.}

\affiliation{University of Southern California\thanksmark{m1} and CREST, ENSAE\thanksmark{m2}}

\address{Department of Mathematics\\
University of Southern California\\
Los Angeles, CA 90089 \\
\printead{e1}}
%\phantom{E-mail:\ }\printead*{e2}}

\address{CREST (UMR CNRS 9194), ENSAE\\
5, av. Henry Le Chatelier, 91764 Palaiseau, France\\
\printead{e2};\quad \printead*{e3}}
%\printead{e3}\\
%\printead{u1}}
\end{aug}

\begin{abstract}
%Consider a Bipartite Stochastic Block Model (BSBM) on vertex sets $V_1$ and $V_2$. 
We establish sufficient conditions of exact and almost full recovery of the node partition
in Bipartite Stochastic Block Model (BSBM) 
using polynomial time algorithms. 
First, we improve upon the known conditions of almost full recovery by spectral clustering algorithms in BSBM. Next, we propose a new computationally simple and fast procedure achieving exact recovery under milder conditions than the state of the art. 
Namely, if the vertex sets $V_1$ and $V_2$ in BSBM have sizes $n_1$ and $n_2$,  we show that the condition
$
p = \Omega\left(\max\left(\sqrt{\frac{\log{n_1}}{n_1n_2}},\frac{\log{n_1}}{n_2}\right)\right)
$
on the edge intensity $p$
is sufficient for exact recovery witin $V_1$.
This condition exhibits  an elbow at $n_{2} \asymp n_1\log{n_1}$ between the low-dimensional and high-dimensional regimes.
The suggested procedure is a variant of Lloyd's iterations initialized with a well-chosen spectral estimator leading to what we expect to be the optimal condition for exact recovery in BSBM. {The optimality conjecture is supported by showing that, for a supervised oracle 
procedure, such a condition is necessary to achieve exact recovery.} The key elements of the proof techniques are different from classical community detection tools on random graphs.
Numerical studies confirm our theory, and show that the suggested algorithm is both very fast and achieves {almost the same}  performance as the supervised oracle. Finally, using the connection between planted satisfiability problems and the BSBM, we improve upon the sufficient number of clauses to completely recover the planted assignment. 
\end{abstract}

%\begin{keyword}[class=MSC]
%\kwd[Primary ]{60K35}
%kwd{60K35}
%\kwd[; secondary ]{60K35}
%\end{keyword}

%\begin{keyword}
%\kwd{sample}
%\kwd{\LaTeXe}
%\end{keyword}

\end{frontmatter}

\section{Introduction}

Unsupervised learning or clustering is a recurrent problem in statistics and machine learning. Depending on the objects we wish to classify, we can generally consider two approaches: either the observed objects are individuals without any interaction, which is often described by a mixture model, or the observed objects are individuals with interactions, which is described by a graph model. In the latter case, the individuals correspond to vertices of the graph and two vertices are connected if the two corresponding individuals interact. The clustering problem becomes then a node clustering problem, which means grouping the individuals by communities. The most known and studied framework for node clustering is  the Stochastic Block Model (SBM), cf. \cite{holland1983stochastic}. In this paper, we focus on the Bipartite Stochastic Block Model (BSBM), cf. \cite{feldman2015subsampled}, which is a non-symmetric generalization of the SBM. This model arises in several fields of applications. For example, it can be used to describe different types of interactions; documents/words \cite{word_object,word}, genes/genetic sequences \cite{gene,gene2} and objects/users in recommendation systems \cite{reco}. Some other examples are related to random computational problems with planted solutions such as planted satisfiability problems, cf. \cite{feldman2018complexity} for a general definition. As shown in \cite{feldman2015subsampled}, three planted satisfiability problems reduce to solving the BSBM. Namely, this concerns planted hypergraph partitioning, planted random $k-$SAT, and Goldreich's planted CSP. Planted satisfiability can be viewed as a $k-$uniform hypegraph stochastic block model. The corresponding reduction to BSBM is characterized by a high imbalance between its two dimensions. For instance, one dimension is $n$ while the other is $n^{r-1}$, where $n$ is the number of boolean literals and $r$ (that can be large) is the distribution complexity of the model that we define later.

\subsection{Definition of Bipartite Stochastic Block Model}

%\begin{defi}[Bipartite Stochastic Block Model]

Let $n_{1+}$, $n_{1-}$, $n_{2+}$ and $n_{2-}$ be four integers such that $n_1 := n_{1+} + n_{1-} \leq n_{2+} + n_{2-} := n_2$, where $n_1\ge 2$, $n_2\ge2$, and let $\delta \in (0, 2)$, $p \in \left(0,1/2\right)$. Consider two sets of vertices $V_1$ and $V_2$ such that:
\begin{itemize}[label=$\--$]
\item $V_1$ is composed of $n_{1+}$ vertices with label $+1$ and of $n_{1-}$ vertices with label $-1$;
\item $V_2$ is composed of $n_{2+}$ vertices with label $+1$ and of $n_{2-}$ vertices with label $-1$. 
\end{itemize}

  We denote by $\sigma(u)\in \{-1,1\}$ the label corresponding to vertex $u$. {We call $|n_{1+} - n_{1-}|/n_1$ (respectively, $|n_{2+} - n_{2-}|/n_2$) the imbalance of the set $V_1$ (respectively, $V_2$).
  In what follows, it is assumed that there exist $\gamma_i\in (0,1), i=1,2,$ such that
  \begin{equation}\label{imb}
   |n_{1+} - n_{1-}|/n_1\le \gamma_1, \qquad
   |n_{2+} - n_{2-}|/n_2\le \gamma_2.
  \end{equation}}
Let $A$ denote the biadjacency matrix, i.e., a rectangular matrix of size $n_1 \times n_2$ whose entries $A_{ij}$ take value 1 if the two corresponding vertices $i\in V_1$ and $j\in V_2$ are connected and take value $A_{ij}=0$ otherwise. 

We say that matrix $A$ is drawn according to the $BSBM(\delta, n_{1+}, n_{1-}, n_{2+}, \allowbreak n_{2-},p)$ model if the entries $A_{ij}$ are independent and
\begin{itemize}[label=\textbullet]
\item $A_{ij} \sim Ber(\delta p)$ if $\sigma(i)=\sigma(j)$, i.e., two vertices $i \in V_1$, $j \in V_2$ with the same label are connected with probability $\delta p$; 
\item $A_{ij} \sim Ber((2-\delta) p)$ if $\sigma(i)\neq\sigma(j)$, i.e., two vertices $i \in V_1$, $j \in V_2$ with different labels are connected with a probability $(2-\delta) p$. 
\end{itemize}

Here, $Ber(q)$ denotes the Bernoulli distribution with parameter $q \in (0,1)$.

In this definition, $p$ represents the overall edge density. The Bipartite SBM is a generalization of the SBM in the sense that we obtain the SBM if $V_1=V_2$. Another possible definition of BSBM is obtained by fixing only $n_1$ and $n_2$ and letting $n_{1+}, n_{1-}, n_{2+}, n_{2-}$ be random variables such that the expectations of $n_{i+}$ and $n_{i-}$ are both equal to $n_i/2$ for $i=1,2$ (then the partitions are called balanced). This is the case when the labels are independent Rademacher random variables as assumed, for example, in the previous work \cite{feldman2015subsampled,florescu2016spectral}.

\subsection{Recovery of communities}

Assume that we observe a biadjacency matrix $A$ drawn according to 
a $BSBM(\delta, n_{1+}, n_{1-}, n_{2+}, n_{2-}, p)$ model. We consider the problem of recovering the node partition associated with $V_1$ %, which is the set of vertices of smaller size, 
from the observation of  $A$. Denote by $\eta_1 \in \{\pm 1\}^{n_1}$ the vector of vertex labels in $V_1$.
Recovering the node partition of $V_1$ is equivalent to retrieving either $\eta_1$ or $-\eta_1$.

%labels and $\eta_i$ the corresponding $i$\up{th} component for $i=1,\dots,n_1$.  \\

As estimators of $\eta_1$ we consider any measurable functions $\hat{\eta}$ of $A$ taking values in $\{ \pm 1 \}^{n_1}$. We characterize the loss of any such estimator $\hat{\eta}$ by the {$\ell_1$-distance between $\hat{\eta}$ and $\eta_1$, that is, by twice} the number of positions at which $\hat{\eta}$ and $\eta_1$ differ:
\begin{align*}
|\hat{\eta}-\eta_1| :=\sum_{i=1}^{n_1}\left|\hat{\eta}_{i}-\eta_{1i}\right|=2 \sum_{i=1}^{n_1} \mathds{1}\left(\hat{\eta}_{i} \neq \eta_{1i}\right),
\end{align*}
where $\hat{\eta}_i$ and $\eta_{1i}$ denote the $i$th components of $\hat{\eta}$ and $\eta_1$, respectively. Since for community detection it is enough to determine either $\eta_1$ or $-\eta_1$
we consider the loss 
$$
r(\eta_1, \hat{\eta}) = \min_{\nu\in \{-1,1\}}|\hat{\eta}-\nu\eta_1|.
$$
The performance of an estimator $\hat{\eta}$ is characterized by one of the three properties defined below. 
{The limits in the following definitions and everywhere in the sequel are considered over sequences $n_1\to \infty$ such that the first imbalance condition in \eqref{imb} is satisfied for every $n_1$. The size of the second set of vertices $n_2$ need not tend to infinity and should only satisfy the second condition in \eqref{imb}. Since we consider the asymptotics as $n_1\to \infty$ the values $\gamma_i$, $p$ and $\delta$ are allowed to depend on $n_1$.}%, $n_2$.
\begin{defi}[weak recovery]\label{weakR}
The estimator $\hat{\eta}$ achieves weak recovery of $\eta_1$
%is possible for the Bipartite Stochastic Block Model $BSBM(\delta, n_1, n_2, p)$ 
if there exists $\alpha \in (0,1)$ such that 
%for every $(V_1, V_2, A)$ drawn according to $BSBM(\delta, n_1, n_2, p)$, there exists an estimator $\hat{\eta} \in \{ \pm 1\}^{n_1}$ of $\eta$ such that for all $\alpha \in (0,1)$:
\begin{align}\label{eq:weakR}
\lim_{n_1\to\infty}\sup_{BSBM}\mathbb{P}\left(\frac{r(\eta_1, \hat{\eta})}{n_{1}}\ge \alpha\right) = 0,
\end{align}
where $\sup_{BSBM}$ denotes the maximum over all distributions of $A$ drawn from $BSBM(\delta, n_{1+}, n_{1-}, n_{2+}, n_{2-}, p)$.
\end{defi}
\noindent {If the communities in $V_1$ are balanced weak recovery with small $\alpha$ can be interpreted as the fact that $\hat{\eta}$ recovers the vertices better than chance. However, if there is a strong imbalance, Defintion \ref{weakR} does not necessarily characterize good estimators as one can achieve weak recovery with small $\alpha$ using a trivial estimator that assigns all vertices to one community. % Under the small imbalance regime, Definition \ref{weakR} remains meaningful. We also note that, 
In this paper, the property stated in Definition \ref{weakR} is not of interest on its own but rather as an auxiliary fact that we need to prove exact recovery. Namely, the initialization of the algorithm proposed below should satisfy the property stated in Defintion \ref{weakR}.
}
\begin{defi}[almost full recovery]
The estimator $\hat{\eta}$ achieves almost full recovery of $\eta_1$ if for all $\alpha \in (0,1)$ we have
\begin{align*}
\lim_{n_1\to\infty}\sup_{BSBM}\mathbb{P}\left(\frac{r(\eta_1, \hat{\eta})}{n_1}\ge \alpha\right) = 0.
\end{align*}
\end{defi}
Almost full recovery means that $\hat{\eta}$ correctly classifies almost every vertex with high probability.

\begin{defi}[exact recovery]
The estimator $\hat{\eta}$ achieves exact recovery of $\eta_1$ if
\begin{align*}
\underset{n_{1}\to \infty}{\lim}\inf_{BSBM} \mathbb{P} \big (r(\eta_1, \hat{\eta}) = 0 \big) = 1.
\end{align*}
\end{defi}
Exact recovery means that $\hat{\eta}$ correctly classifies all the vertices 
with high probability.
%We may extend those definitions to the case where we consider the narrowed definition of the Bipartite SBM.
 %Recall that $n_{1} \leq n_2$ by the definition of the BSBM model. Hence, as the limits in Definitions 1 -- 3 are taken as $n_{1}$ grows to infinity, the same holds for $n_2$. Also, as these definitions are asymptotic they assume that the values $p$ and $\delta$ are allowed to depend on $n_1$, $n_2$.
%All proofs are deferred to the Appendix. \\
%\begin{itemize}
%\item clustering pb : often NP-hard
%\item several of modelling such problems (graph or mixture models)
%\item for graphs : most studied clustering problem $\rightarrow$ partition nodes of a SBM graph
%\item generalisation of SBM : BSBM
%\item two ways of theoretically improving the present results : speed and conditions over density
%\item concerning the technical results : efficient algorithm
%\item study of the BSBM : state of the art
%\end{itemize}
\subsubsection{Notation}

We will use the following notation. For given sequences $a_n$ and $b_n$, we write that $a_n=O(b_n)$ (respectively, $a_n = \Omega(b_n)$) if there is an absolute constant $c$ such that $a_n \leq c b_n$ (respectively, $a_n \geq c b_n$). We write $a_n \asymp b_n$ if  $a_n=O(b_n)$ and $a_n = \Omega(b_n)$. For $a,b \in \mathbb{R}$, we denote by $a \vee b$ (respectively, $a \wedge b$) the maximum (respectively, minimum) of $a$ and $b$. For $x, y \in \mathbb{R}^m$ for any $m \in \mathbb{N}$, we denote by $x^\top y$ the Euclidean scalar product, by $\| x \|_2$ the corresponding norm of $x$ and by sign($x$) the vector of signs of the components of $x$. For any matrix $M \in \mathbb{R}^{m \times m}$,  we denote by $\| M \|_{\infty}$ its spectral norm. Further, $\mathbf{I}_m$ denotes the ${m \times m}$ identity matrix  and $\mathds{1}(\cdot)$ denotes the indicator function. 
We denote by $c$ positive constants that may vary from line to line.

\section{Reduction to a spiked model }

The biadjacency matrix $A$ can be written as
\begin{align*}
A = \mathbb{E}(A) + W
\end{align*}
where $A$ is observed, $\mathbb{E}(A)$ is interpreted as the signal, and $W:=A - \mathbb{E}(A)$ as the noise.  
It is easy to check that
\begin{align}\label{ea}
\mathbb{E}(A) = p \mathds{1}_{n_{1}}\mathds{1}_{n_{2}}^{\top} + (\delta-1)p \eta_{1}\eta_{2}^{\top},
\end{align}
where $\mathds{1}_{n_{1}}$ (respectively, $\mathds{1}_{n_{2}}$) is the vector of ones with dimension $n_{1}$ (respectively, $n_{2}$) and $\eta_{1},\eta_{2}$ are the vectors of labels corresponding to the sets of vertices $V_{1}$ and $V_{2}$, respectively. The second component on the right hand side of \eqref{ea}  contains information about the vector $\eta_{1}$ that we are interested in, while the first component 
$p\mathds{1}_{n_{1}}\mathds{1}_{n_{2}}^{\top}$ is non-informative about the labels. Assuming parameter $p$ to be known we can simply subtract this component from $A$. From an adaptive perspective, one way to eliminate the non-informative component is by getting an estimator $\hat{p}$ of $p$, then considering $A-\hat{p}\mathds{1}_{n_{1}}\mathds{1}_{n_{2}}^{\top}$ as the new data matrix. Another way to disregard this component is to assume, as in \cite{feldman2015subsampled,florescu2016spectral}, that the partitions are balanced, which implies the orthogonality of $\mathds{1}_{n_{i}}$ and $\eta_i$ for $i=1,2$. This assures that $\eta_1$ and $\eta_2$ are the  singular vectors of $\mathbb{E}(A)$ corresponding to the second largest singular value, which makes it possible to recover them  with suitable accuracy from the observation of $A$.

In this paper, we follow the first approach where we estimate $p$ by
\begin{equation}\label{eq:estim_p}
\hat{p} = \frac{1}{n_1n_2} \mathds{1}^\top_{n_1}A\mathds{1}_{n_2}.
\end{equation}
Then we consider the corrected adjacency matrix 
\begin{align}\label{spiked}
\hat{A}:= A - \hat{p} \mathds{1}_{n_1}\mathds{1}^\top_{n_2} = (\delta-1)p \eta_{1}\eta_{2}^{\top} + \underbrace{ W + (p-\hat{p}) \mathds{1}_{n_1}\mathds{1}^\top_{n_2}}_{noise}.
\end{align}
This is a special case of spiked matrix model where the underlying signal and the noise have a particular structure. 
In the rest of this paper, we assume that the observed matrix $\hat{A}$ is of the form \eqref{spiked}. %, akin to the previous work on BSBM (cf. \cite{feldman2015subsampled}, \cite{florescu2016spectral}). 
%Note that the assumption of known $p$ that enables this reduction is very mild since estimating $p$ is much easier than estimating the vector $\eta_{1}$. %Another alternative to remove this part from the mean is through considering the second eigenvector of the adjacency matrix. 

A well-known approach to community detection is the spectral approach, i.e., clustering according to the signs of the entries of eigenvectors or singular vectors of the adjacency matrix or its modified version. 
In our case, $\eta_1$ is the left singular vector associated with the largest singular value of the signal matrix $(\delta-1)p \eta_{1}\eta_{2}^{\top}$. Since $\mathbb{E}(\hat{A})$ is unknown -- only $\hat A$ is observed --  a natural algorithm for recovering $\eta_1$ would, at first sight, consist in computing the left singular vector of $\hat A$ corresponding to the top singular value and then taking the signs of the entries of this vector as estimators of the entries of $\eta_1$. However, such a method provides a good estimator of $\eta_1$ only if the top singular value $(\delta-1)p$ of the signal matrix is much larger than the spectral norm of the noise term in \eqref{spiked} that is dominated by the spectral norm of $W$ under mild assumptions on the imbalance $\gamma_1\gamma_2$. As noticed in \cite{florescu2016spectral}, this approach suffers from a strict deterioration of sufficient conditions of recovery when $n_2$ grows larger than $n_1$. The problem can be avoided by applying the spectral approach to \textit{hollowed matrix} $H(\hat{A}\hat{A}^\top)$ rather than to $\hat{A}$, where $H : \mathbb{R}^{n_{1}\times n_{1}} \rightarrow \mathbb{R}^{n_{1}\times n_{1}} $ is the linear operator defined by the relation
\begin{align*}
H(M) = M - \diag(M),~\forall M \in \mathbb{R}^{n_{1}\times n_{1}}. 
\end{align*}
Here, $\diag(M)$ is a diagonal matrix with the same diagonal as $M$. 
The corresponding spectral estimator of $\eta_1$ is 
\begin{align}
\label{eq:eq2}
\eta_1^0 = \sign(\hat{v}),
\end{align}
where $\hat{v}$ is the eigenvector corresponding to the top eigenvalue of $H(\hat A \hat A^\top)$. We will further refer to $\eta_1^0$ as {\it spectral procedure on hollowed matrix}. The properties of $\eta_1^0$ are studied in Section \ref{sec:sec1}. In particular, we show that $\eta_1^0$ achieves  almost full recovery under milder conditions than previously established in \cite{florescu2016spectral} for a different method called the diagonal deletion SVD. However, it is not known whether $\eta_1^0$ can achieve exact recovery. 

In order to grant exact recovery, we propose a new estimator. Namely, we run the sequence of iterations $(\hat{\eta}^k)_{k \geq 1}$ defined by the recursion

\begin{align}\label{eq:eq4}
\hat{\eta}^{k+1} = \sign \big( H(\hat{A}\hat{A}^\top) \hat{\eta}^k \big),~\quad k= 0,1,\dots,
\end{align}
with the spectral estimator as initializer: $\hat{\eta}^{0}=\eta_1^0$. Our final estimator is  $\hat{\eta}^{m}$ with $m> \frac{\log n_1}{2\log 2}-\frac{3}{2}$. 
We call this procedure the {\it hollowed Lloyd's algorithm}. It is inspired by Lloyd's iterations, whose statistical guarantees were studied in the context of SBM and Gaussian Mixture Models by \cite{lu2016statistical}. More recently, this approach was used in \cite{ndaoud2018sharp} to derive sharp optimal conditions for exact recovery in the Gaussian Mixture Model. It follows from those papers that the issue of proper initialization of Lloyd's algorithm is essential. The question of proving optimality of recovery by Lloyd's algorithm under random initialization is still open, both in the Gaussian Mixture Model and in the BSBM model.

\section{Related work}
While the literature about the classical SBM abounds (we refer to the  paper \cite{abbe2015exact} and references therein), fewer results are known about the Bipartite SBM.    Papers \cite{amini1,amini2,neuman} consider  more general BSBM settings than ours. Being specified to our setting, their results { guarantee consistency for clustering  under conditions not  covering the high-dimensional regime $n_2\gg n_1$}.
In particular, paper \cite{amini1} shows that consistency can be achieved by spectral clustering on an appropriately regularized adjacency matrix when  $n_2\asymp n_1$. %The main theorem in \cite{amini2} requires $p$ to be greater than $(n_1n_2)^{-1/2}$ and $n_1p \to \infty$, which again is not possible ({\color{red} \bf faux, car les deux conditions sont compatibles! Autre version: 
As an example of limitations used in \cite{amini2}, we refer to the main theorem in \cite{amini2} (Theorem 1) that requires $p^2 = O(n_1/n_2^2)$ in our setting (cf. assumption (A3) in \cite{amini2}). This assumption combined with the necessary condition for weak recovery  $p^2 = \Omega((n_1n_2)^{-1})$
 only allows for values of $n_1,n_2$ such that $n_2 = O(n_1^2)$. 
 In \cite{neuman}, the focus is on handling mutiple and possibly overlapping clusters. The recovery conditions from \cite{neuman} being specified to our setting (two non-overlapping clusters) are far from optimal. %Regarding the considered thresholds for $p$, the results in \cite{amini1,amini2,neuman} are not comparable to our work  since we cover the regime $n_2 \gg n_1$.
 
On the other hand, papers \cite{feldman2015subsampled,florescu2016spectral,cai2019subspace} {study the problem on finding proper thresholds for $p$ under conditions covering the high-dimensional regime $n_2 \gg n_1$. } In particular, \cite{florescu2016spectral} proves that the sharp phase transition for the weak recovery problem occurs around the critical probability $p_{c} = \frac{(\delta-1)^{-2}}{\sqrt{n_{1}n_{2}}}$. The sufficient condition in this case is based on a reduction to SBM then using any optimal ``black-box'' algorithm for detection in the SBM as in \cite{bordenave,massoulie,mossel}. 

For the problem of exact recovery, \cite{feldman2015subsampled} obtained what we will further call the state of the art sufficient conditions. Namely, using the Subsampled Power Iteration algorithm, \cite{feldman2015subsampled} shows that the condition $p = \Omega\left(\frac{(\delta-1)^{-2}\log{n_1}}{\sqrt{n_1n_2}}\right)$ is sufficient to achieve exact recovery. {Although no necessary condition for this property is known, it is conjectured in \cite{feldman2015subsampled} that $p = \Omega\left(\frac{\log{n_1}}{\sqrt{n_1n_2}}\right)$ is
%at least $\Omega\left( \sqrt{n_1n_2}\log{n_1}\right)$ edges are 
necessary for exact recovery. Our results below disprove this conjecture.}    

%By analogy to spiked models, one may expect spectral algorithms to achieve optimal recovery in the BSBM.  
Spectral algorithms for BSBM were investigated in \cite{florescu2016spectral,cai2019subspace}. In particular, \cite{florescu2016spectral} compared sufficient conditions of almost full recovery for the classical SVD algorithm and for the diagonal deletion SVD. It was shown in \cite{florescu2016spectral} that, in the 
high-dimensional regime $n_2 \gg n_1$, the diagonal deletion SVD provides a strict improvement over the classical SVD. One way to explain this improvement is by observing that, in this regime, the spectral norm of the expectation of the noise term $WW^{\top}$ is much larger than its deviation. It was proved in \cite{florescu2016spectral} that $p=\Omega\left( \frac{\log{n_1}}{\sqrt{n_1n_2}}\right)$ is sufficient to achieve  almost full recovery through the diagonal deletion SVD algorithm. %\simo{The latter result was further improved in \cite{cai2019subspace}, where under similar conditions it was shown that diagonal SVD actually achieves exact recovery.} 
Note that \cite{feldman2015subsampled} proved that, under similar conditions, the  Subsampled Power Iteration algorithm achieves a better result, i.e., it provides exact recovery rather than almost full recovery. {The most recent paper  \cite{cai2019subspace} parallel to our work shows 
that the diagonal SVD also upgrades from almost full to exact recovery under the conditions that are analogous to \cite{florescu2016spectral} for moderate $n_2\ge n_1$ but deteriorate for very large $n_2$ (for example, if $n_2\asymp e^{n_1}$).} 
The results of \cite{feldman2015subsampled,florescu2016spectral,cai2019subspace} 
are summarized in Table 1.%\ref{table1}.  
\begin{center}
\scalebox{1}{\begin{tabular}{|c|c|c|c|}%\label{table1}
    \hline 
     & & & \\
   Ref. & Results & Conditions & Algorithm \\
     & & & \\
   \hline 
        & & & \\
   \cite{feldman2015subsampled} &
   \makecell{Exact \\ recovery}
   & $\left\{
    \begin{array}{ll}
        { \text{known}\ p}, \, n_2 \geq n_1,  \\
        p \geq C (\delta-1)^{-2} \frac{\log n_1}{\sqrt{n_1 n_2}}
    \end{array}
\right.$
& \makecell{Subsampled \\ iterations} \\
     & & & \\
   \hline 
        & & & \\
   \cite{cai2019subspace} &
   \makecell{Exact \\ recovery}
   & $\left\{
    \begin{array}{ll}
         { \text{known}\ p ,} \\
        p \geq C (\delta-1)^{-2} \left(\frac{\log (n_1+n_2)}{\sqrt{n_1 n_2}}\vee \frac{\log{(n_1+n_2)}}{n_2} \right)
    \end{array}
\right.$
& \makecell{Diagonal \\ deletion SVD} \\
     & & & \\
\hline 
     & & & \\
\cite{florescu2016spectral}
& \makecell{Almost full \\ recovery}
& $\left\{
    \begin{array}{ll}
        { \text{unknown}\ p},\, n_2 \geq n_1 (\log n_1)^4 , \,\gamma_1=\gamma_2=0,\\
        p \geq C_{\delta} \frac{\log n_1}{\sqrt{n_1 n_2}}
    \end{array}
\right.$
& \makecell{Diagonal \\ deletion SVD} \\
     & & & \\
\hline 
     & & & \\
\cite{florescu2016spectral}
& \makecell{Weak \\ recovery}
& $\left\{
    \begin{array}{ll}
        { \text{unknown}\ p},\, n_2 \geq n_1,\, \gamma_1=\gamma_2=0,\\
        p > \frac{(\delta-1)^{-2}}{\sqrt{n_1n_2}}
    \end{array}
\right.$
& \makecell{SBM \\ reduction} \\
     & & & \\
     \hline 
\end{tabular}}
\captionof{table}{Summary of the results of  \cite{feldman2015subsampled,florescu2016spectral,cai2019subspace}. Here, $C_{\delta}>0$ is a constant depending on $\delta$. {In this table, condition $\gamma_1=\gamma_2=0$ means that the labels are independent Rademacher random variables.} }
\end{center}
{We emphasize that \cite{feldman2015subsampled,florescu2016spectral} focus on the regime $n_2 \ge n_1$. It will be also the main challenge in the present paper even though our results are valid for all $n_1,n_2$ with no restriction. There are two reasons for that:
\begin{itemize}
    \item The high-dimensional regime $n_2 \gg n_1$ is of dominant importance in the applications of bipartite SBM.
    \item The high-dimensional regime is challenging from the theoretical point of view. The case $n_2 \leq n_1$ is more direct to handle as it can be solved similarly to standard SBM. Indeed,  optimal results for $n_2 \leq n_1$ are achieved by SVD type algorithms applied to the adjacency matrix $A$ (cf. \cite{amini1,amini2}). They are based on a control of the spectral norm of $W$. While the behavior of the spectral norm of $W$ is well understood (cf. \cite{vanhandel}), existing results for the spectral norms of $WW^{\top} - \mathbb{E}(WW^{\top})$  or of $H(WW^{\top})$ that one needs to control when $n_2\gg n_1$ turn out to be suboptimal. It makes the regime  $n_2\gg n_1$ quite challenging.
\end{itemize}

}

  Under the condition $n_2 \geq n_1\log{n_1}$, the state of the art results can be summarized by the following diagram leaving open the optimal value $p=\bf{p}^*$ at which exact recovery can be achieved. 

\begin{tikzpicture}

% horizontal axis
\draw[->] (0,0) -- (10,0) node[anchor=north] {$p$};
% labels
\draw	(0,0) node[anchor=north] {0}
		(3,0) node[anchor=north] {$\frac{(\delta-1)^{-2}}{\sqrt{n_1n_2}}$}
		(5,0) node[anchor=north] {$\bf{p}^{*}$}
		(7,0) node[anchor=north] {$\frac{(\delta-1)^{-2}\log{n_1}}{\sqrt{n_1n_2}}$};
% ranges
\draw	(1.5,1) node{\bf{\color{red}weak recovery }};
\draw	(1.5,0.5) node{\bf{\color{red} is impossible}};

\draw	(8.5,1) node{\bf{\color{blue}exact recovery  }};
\draw	(8.5,0.5) node{\bf{\color{blue} is possible }};

% vertical axis
%\draw[->] (0,0) -- (0,4) node[anchor=east] {$\Phi_{o}(a)$};
% nominal speed
\draw[dotted] (3,0) -- (3,2);
\draw[dotted] (7,0) -- (7,2);
\draw[thick] (5,0) -- (5,2);

% Us
%\draw[thick] (0,3) -- (4,3);
%\draw[thick] (6,1) -- (10,1);%label

% Psis
%\draw[thick,dashed]  (4,3) parabola[bend at end] (6,1); %label

\end{tikzpicture}

A related recent line of work developed optimal clustering algorithms for Gaussian Mixture Models (GMM) \cite{lu2016statistical,giraud2019partial,ndaoud2018sharp,lof}. It was shown in \cite{lu2016statistical} that clustering with optimality properties in GMM can be achieved by an iterative algorithm analogous to Lloyd's procedure. Moreover, \cite{ndaoud2018sharp} proved that a version of such iterative clustering algorithm attains the sharp phase transition for exact recovery in those models. Based on an analogy between the GMM and the BSBM, it is conjectured in \cite{ndaoud2018sharp} that similar algorithms can achieve almost full recovery and exact recovery in bipartite graph models. %,  leaving open questions about the latter concerning the characterisation of a sharp phase transition and the possibility of using . 
Namely, comparing the first two moments of the matrices arising in the two models one may expect $p=\Omega\left((\delta-1)^{-2}\sqrt{\frac{\log{n_1}}{n_1n_2}}\right)$ to be sufficient to achieve exact recovery in the BSBM, provided that $n_2 \geq n_1\log{n_1}$. {This heuristics suggests a logarithmic improvement over the state of the art condition presented in the diagram above. More interestingly, %suggests that, if $n_2 \geq n_1\log{n_1}$, we only need $\Omega\left( \sqrt{n_1n_2\log{n_1}}\right)$ edges to exactly recover the partition of $V_1$, which 
it goes against another, seemingly more natural, heuristics  based on an analogy with standard SBM and conjecturing the right recovery condition in the form $p=\Omega\left(\frac{\log{n_1}}{\sqrt{n_1n_2}}\right)$ 
(cf. \cite{feldman2015subsampled}). We show below that, surprisingly, the analogy with  GMM and not with SBM (however, the ``closest parent" of BSBM) appears to be correct. }

Finally, some consequences were obtained for planted satisfiability problems.
Reduction of those problems  to BSBM allows one to get sufficient conditions of complete recovery of the planted assignment. We refer to \cite{feldman2015subsampled} for the details of this reduction. Namely, it is shown in \cite{feldman2015subsampled} that considering a planted satisfiability problem is equivalent to considering a BSBM where $n_1 = n$ and $n_2 = n^{r-1}$, with $n$ and $r \geq 2$ defined below. For any satisfiability
problem, we are interested in $m$, which is the sufficient number of $k$-clauses from $C_{k}$ in order to recover completely
the planted assignment $\sigma$. Here, $C_k$ is the set of all ordered $k$-tuples of $n$ literals $x_1,\dots,x_n$ and their negations with
no repetition of variables. For a $k$-tuple of literals $C$ and an assignment $\sigma \in \{-1,+1\}^n$, $\sigma(C)$ denotes the vector of values that $\sigma$ assigns to the literals in $C$.  Given a planting distribution $Q:\{-1,+1\}^k \to [0,1]$, and an assignment $\sigma$, we define the random constraint satisfaction problem $F_{Q,\sigma}(n,m)$ by drawing $m$ $k$-clauses from $C_k$ independently according to the distribution
$$
Q_{\sigma}(C) = \frac{Q(\sigma(C))}{\sum_{C'\in C_k}Q(\sigma(C'))}.
$$
A related class of problems is one in which for some fixed predicate $P:\{-1,1\}^{k}\to\{-1,1\}$, an instance is generated by choosing a planted assignment $\sigma$ uniformly at random and generating a set of $m$ random and uniform $P$-constraints. That is, each constraint is of the form $P(x_{i_1},\dots,x_{i_k})$ = P($\sigma_{i_1},\dots,\sigma_{i_k}$), where $(x_{i_1},\dots,x_{i_k})$ is a randomly and uniformly chosen $k$-tuple of variables (without repetitions).

In simpler words $m$ plays the role of $pn_1n_2$ in the BSBM, and any sufficient condition on $p$ leads to a sufficient condition for $m$. It was shown in \cite{feldman2015subsampled}  that the following conditions
are sufficient to achieve exact recovery in some of the satisfiability problems.
\begin{itemize}
    \item For any planting distribution $Q:\{-1,1\}^{k}\to[0,1]$, there exists an algorithm that for any assignment $\sigma \in \{-1,1\}^n$, given an instance of $F_{Q,\sigma}(n,m)$, completely recovers the planted assignment $\sigma$ for $m=O(n^{r/2}\log{n})$. Here, $r\geq 2$ is the smallest integer such that there is some $S\subseteq \{1,\dots,k\}$ with $|S|=r$, for which the discrete Fourier coefficient $\hat{Q}(S)$ is non-zero.
    \item For any predicate $P:\{-1,1\}^{k}\to\{-1,1\}$, there exists an algorithm that for any assignment $\sigma$, given $m$ random $P$-constraints, completely recovers the planted assignment $\sigma$ for $m=O(n^{r/2}\log{n})$ where $r\geq 2$ is the degree of the lowest-degree non-zero Fourier coefficient of $P$.
\end{itemize}

\section{Main contributions}
Our findings can be summarized as follows.
\begin{itemize}
\item {\it Exact recovery.} We present a novel method - the {\it hollowed Lloyd's algorithm} - that achieves exact recovery under strictly milder conditions than the state of the art. {Namely, we show that 
\begin{equation}\label{pthr}
p = \Omega\left(\sqrt{\frac{\log{n_1}}{n_1n_2}} \vee \frac{\log{n_1}}{n_2}\right)
\end{equation}
is sufficient to achieve exact recovery in the BSBM. Condition \eqref{pthr} exhibits  an elbow at $n_{2} \asymp n_1\log{n_1}$ between the low-dimensional and high-dimensional regimes. In the low-dimensional regime $n_{2} \leq n_1\log{n_1}$, it takes the form $p = \Omega\left( \frac{\log{n_1}}{n_2}\right)$, the same as the sufficient condition in \cite{cai2019subspace}, that can be shown minimax optimal using similar lower bound techniques as in \cite{bandeira201518} for SBM (clustering oracle with side information). Such a lower bound was  formalized, for the Bipartite SBM, in Theorem 2 of \cite{amini2} under the conditions $n_1 \leq n_2$ and $n_2 = O( n_1\log{n_1})$ that correspond to assumptions $(A1)-(A4)$ from \cite{amini2}. On the other hand, in the high-dimensional regime $n_{2} \geq n_1\log{n_1}$ the sufficient condition of exact recovery \eqref{pthr} reads as $p = \Omega\left(\sqrt{\frac{\log{n_1}}{n_1n_2}}\right)$.  
We argue that this condition is tight %  based on the GMM reduction mentioned above. 
by showing that even a supervised oracle procedure fails to achieve exact recovery in the regime $n_2 \geq n_1 \log{n_1}$ if $p < c\sqrt{\frac{\log{n_1}}{n_1n_2}}$ for a constant $c>0$ small enough. Importantly, 
 our findings imply that the  condition $p = \Omega\left(\frac{\log{n_1}}{\sqrt{n_1n_2}}\right)$ common for all the related work and  based on the analogy with usual SBM is not necessary for exact recovery in the BSBM when $n_2 \geq n_1 \log{n_1}$.}
\item {\it Almost full recovery.} We provide a new sufficient condition for  almost full recovery by spectral techniques using the diagonal deletion device.  Our spectral estimator and its analysis are different from \cite{florescu2016spectral}, where another diagonal deletion method was suggested.
The analysis uses an adapted version of matrix Bernstein inequality applied to a sum of hollowed rank one random matrices where bounding the corresponding moments, in operator norm, involve combinatorics arguments. This leads to an improvement upon the sufficient condition of \cite{florescu2016spectral}.  We show that, unlike in the Gaussian case, hollowing the Gram matrix yields, both theoretically and empirically, a strict improvement over debiasing, i.e., subtracting the expecation of the Gram matrix. 
\item   The hollowed Lloyd's algorithm that we propose is computationally faster than the previously known methods. Its analysis that we develop is novel and makes it possible to transform any estimator achieving weak recovery into another one achieving exact recovery. We expect this analysis to be useful to solve more general exact recovery problems for random graphs. 
\item In contrast to the related work, where simplifying assumptions of either zero imbalance ($\gamma_1=\gamma_2=0$)  as in \cite{florescu2016spectral} or known $p$ as in \cite{feldman2015subsampled,cai2019subspace} were imposed, %\simo{dans \cite{feldman2015subsampled} la valeur de $p$ est utilisee dans leur algorithme. Ce n'est pas clair s'ils supposent vraiment $\gamma_1=\gamma_2=0$ par contre. Justement quand $p$ est connu on a pas besoin de condition d'imbalance.} 
our approach is more general. In particular, our exact recovery result holds adaptively to $p$ under a mild assumption on $\gamma_1\gamma_2$. Notice that as $\gamma_1\gamma_2$ get closer to $1$, then estimation of $p$ gets harder. Our theoretical findings are supported by numerical experiments, where we show that our iterative procedure (with or without spectral initialization) outperforms spectral methods and achieves almost the same performance as the supervised oracle.
\item Our results regarding 
%\begin{enumerate}[label=(\alph*)]
    %\item almost full recovery based on the spectral estimator,
    %\item 
 almost full recovery based on the spectral estimator,  exact recovery via the hollowed Lloyd's algorithm, and the
    impossibility of exact recovery via the supervised oracle
%\end{enumerate}
are summarized in the table below. \\
%
%\sout{n_2 \geq n_1\log{n_1}},
{
\begin{center}
\scalebox{1}
{\hspace{-1.3cm}
\begin{tabular}{|c|c|c|}
    \hline 
      & & \\
    Results & Conditions & Procedure \\
      & & \\
   \hline 
         & & \\
   \makecell{ Almost full recovery\\  {\it by spectral methods}}
  & $\left\{
   \begin{array}{ll}
         \text{unknown} \ p, \ \ \gamma_1\gamma_2 \leq 1/C'_{n_1} \\
        p \geq C_{n_1} (\delta-1)^{-2} \left(\sqrt{\frac{\log n_1}{n_1 n_2}}\vee \frac{\log{n_1}}{n_2}\right)
    \end{array}
\right.$
%& \makecell{Spectral } \\
& \makecell{Spectral  \\ on hollowed matrix} \\
     & & \\
\hline 
  %    & & \\
\makecell{ Exact recovery\\  }
& $\left\{
    \begin{array}{ll}
    \text{unknown} \ p, \ \
         \gamma_1\gamma_2 < 1/480, \\
        p \geq C (\delta-1)^{-2} \left(\sqrt{\frac{\log n_1}{n_1 n_2}}\vee \frac{\log{n_1}}{n_2}\right)
    \end{array}
\right.$
& \makecell{Hollowed Lloyd's} \\
      & & \\
\hline 
      & & \\
\makecell{Impossibility of \\ exact recovery}
& $\left\{
    \begin{array}{ll}
        n_2 \geq n_1\log{n_1},\ \gamma_1=\gamma_2 = 0,\\
        p < C_{\delta} \sqrt{\frac{\log{n_1}}{n_1n_2}}
    \end{array}
\right.$
& \makecell{Oracle} \\
      & & \\
     \hline 
\end{tabular}
}
%\begin{center}
%\hspace*{-0.5cm}
%\scalebox{1}{
%\begin{tabular}{|c|c|c|}
 %   \hline 
%      & & \\
%      Result & Conditions 1 & Conditions 2\\
%      & & \\
 %  \hline 
 %    & & \\
 %   (a)
 % &   $\left\{
  %  \begin{array}{ll}
  %      n_2 \geq n_1^{1 + \varepsilon} \\ %\log n_1
  %      {p \geq C_{n_1}(\delta-1)^{-2}(1+\frac{1}{\varepsilon}) %\sqrt{\frac{\log n_1}{n_1 n_2}} } \\
 %   \end{array}
%\right.$  
%&  $\left\{
 %   \begin{array}{ll}
  %      n_2 \geq n_1 \log n_{1} \\
  %      p \geq {C_{n_1}} (\delta-1)^{-2}\frac{\log n_1}{\sqrt{n_1 n_2 %\log \log n_1}}
 %   \end{array}
%\right.$ \\
% & &  \\
 %    \hline 
 %    & & \\
 %    (b)
%  &  $\left\{
%    \begin{array}{ll}
  %      {n_2 \geq  n_1^{1+\varepsilon}} %\log n_1}
 %       \\
  %      {p \geq C(\delta-1)^{-2}(1+\frac{1}{\varepsilon}) %\sqrt{\frac{\log n_1}{n_1 n_2}} }
%    \end{array}
%\right.$
%&  $\left\{
%    \begin{array}{ll}
%        n_2 \geq n_1 \log n_{1} \\
%        p \geq C (\delta-1)^{-2}\frac{\log n_1}{\sqrt{n_1 n_2 \log \log n_1}}
%    \end{array}
%\right.$ \\
%      & & \\
%      \hline 
%\end{tabular}
%}
\captionof{table}{Summary of our main contributions. Here, $C_\delta>0$ is a positive constant depending on $\delta$, $C>0$ is an absolute  constant
and $C_{n_1}$, $C'_{n_1}$ are any sequences such that $C_{n_1}, C'_{n_1} \to\infty$ as $n_1\to\infty$.
}
\end{center}
}

\item As a byproduct, we also improve upon sufficient conditions of \cite{feldman2015subsampled} for exact recovery in some of the satisfiability problems. Namely, our results imply the following.
\begin{enumerate}
    \item For any planting distribution $Q:\{-1,1\}^{k}\to[0,1]$, there exists an algorithm that for any assignment $\sigma$, given an instance of $F_{Q,\sigma}(n,m)$, completely recovers the planted assignment $\sigma$ for $m=O(n^{r/2}\sqrt{\log{n}})$ where $r\geq 3$ is the smallest integer such that there is some $S\subseteq \{1,\dots,k\}$ with $|S|=r$, for which the discrete Fourier coefficient $\hat{Q}(S)$ is non-zero.
    \item For any predicate $P:\{-1,1\}^{k}\to\{-1,1\}$, there exists an algorithm that for any assignment $\sigma$, given $m$ random $P$-constraints, completely recovers the planted assignment $\sigma$ for $m=O(n^{r/2}\sqrt{\log{n}})$ where $r\geq 3$ is the degree of the lowest-degree non-zero Fourier coefficient of $P$.
\end{enumerate}
\end{itemize}

\section{Properties of the spectral method}
\label{sec:sec1}

In this section, we analyze the risk of the spectral initializer $\eta_1^0$. As in the case of SDP relaxations of the problem, the matrix of interest is the Gram matrix $\hat A \hat A^\top$. It is well known that it suffers from a bias that grows with $n_2$. In \cite{royer2017adaptive}, a debiasing procedure is proposed using an estimator of the covariance of the noise. In this section, we consider a different approach that consists in removing the diagonal entries of the Gram matrix. %It is worth reminding here that this approach was also used in \cite{florescu2016spectral}. 

%%Let us show that an oracle estimator $\eta^{**}$ of $\eta$ can be defined as:\\

%%\begin{align}
%%\label{eq:eq3}
%%\eta^{**} = \sign \left( H(AA^\top) \eta \right)
%%\end{align}

%We observe that $A v_2 = (\delta -1) p \sqrt{n_1 n_2} v_1 + W v_2$. If $v_2$ was known, and assuming that it is possible to control the spectral norm of $W v_2$, we could estimate the entries of $\eta$ by taking the sign of $A v_2$. However, in practice, $v_2$ is unknown, and even more difficult to estimate than $v_1$. Now, we also observe that $v_1 ^\top A = (\delta -1)p \sqrt{n_1 n_2} v_2^\top + v_1 ^\top W$. Thus, applying the same reasoning, if $v_1$ was known, and assuming that it is possible to control the spectral norm of $W^\top v_1$, we could estimate $v_2$ (up to a $\sqrt{n_2}$ factor) by taking the sign of $A^\top v_1$. This provides the following oracle estimator $\eta^{**}$ given by :

%\begin{align*}
%\forall i=1,\dots, n_1, \hspace*{0.5cm} \eta_i^{**} = \sign \left(  A_i  \left( \sum_{k \neq i}  \eta_k A_k^\top\right) \right) \text{ where } A = \begin{bmatrix} A_1 \\ \vdots \\ A_{n_1} \end{bmatrix}
%\end{align*}

%In practice, $\eta^{**}$ is only an oracle estimator, since obtaining $\eta_i^{**}$ %requires a prior knowledge of every $\eta_k$ for $k \neq i$.\\

We give some intuition about this procedure when $p$ is known. In this case the adjacency matrix can be replaced by
\[
\tilde A = A - p\mathds{1}_{n_1}\mathds{1}_{n_2}^\top.
\]
The general case follows similarly since one can show that $|\hat{p}-p|$ does not exceed the noise level arising when $p$ is known (see the details below).   The  spectral norm of the expected noise matrix $\mathbb{E}(WW^\top)$ is of the 
order of $n_2 p$.
If $n_{2} \gg n_{1}$, which is the most interesting case in the applications, this is too large compared to the deviation, in the spectral norm, of the noise matrix from its expectation, cf. \cite{florescu2016spectral}.  Since the expectation of the noise $WW^{\top}$ is a diagonal matrix, removing diagonal terms is expected to reduce the spectral norm of the noise and hence to make the recovery problem easier. Specifically, observe that the matrix $H(\tilde A \tilde A^\top)$ can be decomposed as follows: 
\begin{align}
\label{eq:eq1}
  H(\tilde A\tilde  A^\top) & = \underbrace{(\delta -1)^2 p^2 n_2 H(\eta_{1} \eta_{1}^\top) }_{signal} \\
  &\quad + \underbrace{H(W W^\top) + p(\delta-1)H( W\eta_{2}\eta_{1}^{\top} +  \eta_{1}\eta_{2}^{\top}W^\top) 
}_{noise}. \nonumber
%(\delta -1)^2 p^2 n_1 n_2 v_1 v_1^\top + H(W W^\top) + H( W \mathbb{E}(A)^\top + \mathbb{E}(A) W^\top) + p^2 n_1 n_2 H(J_{n_1}) - (\delta -1)^2 p^2 n_1 n_2 I_{n_1} \\
\end{align}
It turns out that the main driver of the noise is $H(WW^\top)$. On the other hand, it is easy to see (cf., e.g., Lemma 17 in \cite{ndaoud2018sharp}) that 
\begin{align}\label{eq:hollow_debias}
\| H(WW^\top) \|_{\infty} \leq 2 \| WW^\top - \mathbb{E}(WW^\top) \|_{\infty}
\end{align}
for any random matrix $W$ with independent columns. This shows that removing the diagonal terms is a good candidate to remove the bias induced by the noise. Thus, diagonal deletion can be viewed as an alternative to debiasing of the Gram matrix. Nevertheless, the operator $H(\cdot)$ may affect dramatically the signal.  Fortunately, it does not happen in our case; the signal term is almost insensitive to this operation since it is a rank one matrix. In particular, we have:
\begin{align*}
\| H(\eta_1 \eta_1^\top) \|_{\infty} = \left(1 - \frac{1}{n_1}\right) \| \eta_1 \eta_1^\top \|_{\infty}.
\end{align*}
Thus, as $n_1$ grows, the signal does not get affected by removing its diagonal terms while we get rid of the bias in the noise term. This motivates the spectral estimator $\eta_1^0$ defined by \eqref{eq:eq2},
where $\hat{v}$ is the eigenvector corresponding to the top eigenvalue of $H(\hat A \hat A^\top)$. %Note that $\eta_1^0$ is different from the spectral estimator of \cite{florescu2016spectral}, which is based on the second eigenvalue of $H(AA^\top)$. This allows us, in particular, to get a correct dependence of the spectral theshold on $\delta$. 
The next theorem gives  sufficient conditions for the  estimator $\eta_1^0$ to achieve
weak and almost full recovery. 
\begin{thm}\label{th:th1}
Let $\eta_1^0$ be the estimator given by \eqref{eq:eq2} with $\hat{p}$ defined in \eqref{eq:estim_p} and let $\alpha \in (0,1)$. Let $(C_{n_1}), (C'_{n_1})$ be sequences of positive numbers that  tend to infinity as $n_1\to \infty$. 

\begin{enumerate}[label=(\roman{enumi})]
    \item 
Let the following conditions hold:

$\left\{
    \begin{array}{ll}
        %\sout{n_2 > n_1\log{n_1},} \\
        \gamma_1\gamma_2 \leq \sqrt{\alpha}/96, \\
        p \geq C(\delta-1)^{-2} \left(\sqrt{\frac{\log n_1}{n_1 n_2}} \vee \frac{\log{n_1}}{n_2} \right),
    \end{array}
\right.$ \\
where $C>C_0/\sqrt{\alpha}$ for an absolute constant $C_0>0$ large enough.
Then the estimator $\eta_1^0$ {satisfies \eqref{eq:weakR}}.

\item Let the following conditions hold:

$\left\{
    \begin{array}{ll}
        %\sout{n_2 > n_1\log{n_1}},\\
        {\gamma_1\gamma_2 \leq 1/C'_{n_1}, }
        \\
       p \geq C_{n_1}(\delta-1)^{-2} \left(\sqrt{\frac{\log n_1}{n_1 n_2}} \vee \frac{\log{n_1}}{n_2} \right).
   \end{array}
\right.$ \\
Then the estimator $\eta_1^0$ achieves almost full recovery of $\eta_1$.
\end{enumerate}
\end{thm}
{ Part (i) of Theorem \ref{th:th1} establishes the property of spectral initializer $\eta_1^0$ that we need to prove the exact recovery in Theorem  \ref{th:th3} below.} {Part (ii) of Theorem \ref{th:th1} improves upon the existing sufficient conditions of  almost full recovery {\it by spectral methods} \cite{florescu2016spectral}, cf. Table 1 above. Theorem \ref{th:th1} covers any $n_1,n_2$ with no restriction, and scales as $\sqrt{\frac{\log n_1}{n_1 n_2}}$ rather than $\frac{\log n_1}{\sqrt{n_1 n_2}}$ in the regime $n_{2}\geq n_{1}\log{n_1}$.} 

The proof of Theorem \ref{th:th1} is given in the Appendix. It is based on a variant of matrix Bernstein inequality applied to a sum of independent hollowed rank one random matrices (Theorem \ref{thm:new_bernstein} in the Appendix). {As a consequence of this new matrix concentration result, we have the  following improved bound for the spectral norm of the noise term.
 \begin{restatable}{prop}{secondlemma}\label{lem:expectation}
Assume that $p \geq C%(\delta-1)^{-2}
\left(\sqrt{\frac{\log n_1}{n_1 n_2}} \vee \frac{\log{n_1}}{n_2} \right)$ for some constant $C>0$. Then, there exists a constant $c_*>0$ %depending only on $\delta$ 
such that
\[
\mathbb{E} \left(\Big\| H\big(WW^\top\big) \Big\|^2_\infty \right)\le  c_*\left( 1+\frac{ n_1\log{n_1}}{n_2}\right)n_1n_2p^2\log{n_1}.
\]
\end{restatable}
On the other hand, for the non-hollowed matrix $WW^\top$, using  matrix Bernstein inequality, we can only obtain that  
\[
\mathbb{E} \left(\Big\| WW^\top \Big\|^2_\infty \right) = O\big((n_1+n_2)^2p^2(\log{n_1})^2\big).
\]
}

Comparing the above two bounds explains why our hollowed spectral method is superior to the standard SVD procedure even in the low-dimensional regime $n_2 = O(n_1\log{n_1})$.

Our next point is to explain why applying the hollowing operator $H(\cdot)$ is better than debiasing by subtraction of $\mathbb{E}(WW^{\top})$. {Inequality~\eqref{eq:hollow_debias} is useful to bound the spectral norm of the hollowed Gram matrix, under the Gaussian Mixture Model (cf. \cite{ndaoud2018sharp}). What is more, one can show that \eqref{eq:hollow_debias} is tight when the noise is isotropic and normal, suggesting that hollowing and debiasing are almost equivalent in the Gaussian Mixture Model. Suprisingly, the same inequality turns out to be loose in the the BSBM model.} It turns out that hollowing the Gram matrix can be strictly better. %than debiasing by subtracting the expectation.
Indeed, the next proposition shows that debiasing the Gram matrix through covariance subtraction can be suboptimal. 

 \begin{restatable}{prop}{firstlemma}\label{lem:suboptimal}
Let %$\gamma_1<1$, $\gamma_2<1$, 
$n_2\geq n_1\log{n_1}$ and 
\begin{equation}\label{brackp}
  18\sqrt{\frac{\log n_1}{n_1 n_2}} \leq p \leq \frac{1}{206 \,n_1\log{n_1}}.   
\end{equation} Then
\[
\mathbb{E}\left(\|WW^{\top} - \mathbb{E}(WW^{\top})\|^2_{\infty} \right) \geq \frac{n_2 p }{40}.
\]
\end{restatable}
Proposition \ref{lem:suboptimal} deals with the high-dimensional regime $n_2\geq n_1\log{n_1}$ under the additional restriction $n_1(\log n_1)^3 = O(n_2)$ that follows from condition \eqref{brackp}.   {Notice that for smaller $p$ satisfying \eqref{brackp}, we  have} 
{$ n_1n_2p^2\log{n_1} = {o}(n_2 p)$}, so that inequality \eqref{eq:hollow_debias} is loose. This explains the suboptimality of  debiased spectral estimator. We further check this fact through simulations in Section \ref{sec:numerical}.  Proposition \ref{lem:suboptimal} also explains why the result in \cite{cai2019subspace} is suboptimal. Indeed, the assumptions in \cite{cai2019subspace} are such that the spectral norm of matrix ${\rm diag}\big(WW^{\top} - \mathbb{E}(WW^{\top})\big)${ (that scales as $\sqrt{n_2 p}$ )} is not bigger in order of magnitude than the spectral norm of the corresponding off-diagonal matrix.
 While this fact is true in several other settings,  it is not in the high-dimensional regime of bipartite clustering, cf. Proposition \ref{lem:suboptimal}.

The question of whether the spectral estimator $\eta_1^0$ can achieve exact recovery under the conditions of Theorem \ref{th:th1} remains open. Pursuing similar arguments as developed in \cite{abbe2017} for the case of SBM would lead to a logarithmic dependence of order $\log{n_1}$ or bigger in the sufficient condition (as it is the case in \cite{cai2019subspace}), and not to the desired $\sqrt{\log{n_1}}$ .
By analogy to the Gaussian Mixture Model, for which it was shown recently in \cite{abbe2020ellp} that the spectral estimator is optimal for exact recovery, we conjecture that the condition $p > C(\delta-1)^{-2} \sqrt{\frac{\log n_1}{n_1 n_2}}$ is sufficient for $\eta_1^0$ to achieve exact recovery whenever $n_{2} \geq n_{1}\log n_{1}$.  Proving such a result would most likely require developing novel concentration bounds for Bernoulli covariance matrices.
%The next section is devoted to construction of an estimator achieving exact recovery under the conditions of part \textit{(i)} of Theorem \ref{th:th1}.

\section{Exact recovery by the hollowed Lloyd's algorithm}

In this section, we present sufficient conditions, under which the hollowed Lloyd's algorithm $(\hat{\eta}^k)_{k \geq 0}$ defined in \eqref{eq:eq4} with spectral initialization achieves exact recovery for all $k$ large enough. 

\begin{thm}\label{th:th3}
Let $(\hat{\eta}^k)_{k \geq 0}$ be the recursion \eqref{eq:eq4} initialized with the spectral 
estimator \eqref{eq:eq2} for $\hat{p}$ given by \eqref{eq:estim_p}. There exists an absolute constant $C>0$ such that if the following conditions hold:
$$\left\{
    \begin{array}{ll}
       %\sout{ n_2 > n_1 \log n_{1},} \\
        \gamma_1\gamma_2 \leq 1/480,\\
        p \geq C (\delta-1)^{-2}\left(\sqrt{\frac{\log n_1}{n_1 n_2 }} \vee \frac{\log{n_1}}{n_2} \right),
    \end{array}
\right.$$
then the estimator $\hat{\eta}^{m}$ with $m=m(n_1) >  \frac{\log n_1}{2\log 2}-\frac{3}{2}$ achieves exact recovery of $\eta_1$.
\end{thm}
Some comments are in order here.
\begin{enumerate}
    \item The approach that we developed to construct $\hat{\eta}^{m}$ is general. In fact, it is a tool that transforms any estimator achieving weak recovery into a new estimator achieving exact recovery under mild assumptions. This can be readily seen from the proof of Theorem \ref{th:th3}. 
    \item Numerically, the procedure $(\hat{\eta}^k)_{k \geq 0}$ considered in Theorem \ref{th:th3} has the same complexity as the spectral initializer $\eta_1^0$. It remains an open question whether  the result of Theorem \ref{th:th3} holds with random initialization, {which would further bring down the complexity}.
    \item We conjecture that the conditions  $ p \geq C (\delta-1)^{-2}\sqrt{\frac{\log n_1}{n_1 n_2 }}$ and $n_2 > n_1 \log n_{1}$ of Theorem \ref{th:th3} cannot be improved. In the next section, we provide a result supporting this fact. The imbalance condition $\gamma_1\gamma_2 = O(1)$ is only required to handle the estimation of $p$. If $p$ is known the results of this paper remain valid with no assumption on $\gamma_1$ and $\gamma_2$.
\end{enumerate}

\section{Impossibility result for a supervised oracle}\label{sec:diagram} Motivated by the spiked reduction of the BSBM model when $p$ is known, we define the supervised oracle as follows 
\[
\tilde\eta_1 = \sign ( H(\tilde A \tilde A^\top )\eta_{1}).
\]
 {Note that this oracle is extremely powerful. We set the definition of the oracle in a compact form using the hollowed matrix $H(\cdot)$ for the purpose of shorter writing. However, if one unfolds this definition, it turns out that the oracle makes a decision about one vertex by using the majority vote of all the other vertices. In other words, this oracle has access to {\it all except one} labels and uses these known labels to predict the remaining one. Specifically, for each label $\eta_{1i}$ to estimate, the supervised oracle has access to the remaining labels $(\eta_{1j}, j\ne i)$ and to $p$. We refer the reader to \cite{ndaoud2018sharp} for more discussion about such an oracle structure. }

We state below an impossibility result corresponding to the supervised oracle.
\begin{restatable}{prop}{firstprop}\label{oracc}
Assume that $n_2 \geq n_1\log{n_1}$ and $\gamma_1 = \gamma_2=0$. There exists $c_\delta>0$ depending only on $\delta$ such that if $p = \sqrt{c_\delta\frac{\log{n_1}}{n_1n_2}}$ then for the oracle $\tilde\eta_1$ we have
\begin{align*}
\underset{n_1 \to \infty}{\lim}  \sum_{i=1}^{n_1}\mathbb{P}(\tilde \eta_{1i} \neq \eta_{1i}) =\infty.
\end{align*}
%\[
%\underset{n_1 \to \infty}{\lim} \mathbb{E}|\eta^{o} - \eta| = \infty.
%\]
\end{restatable}
Proposition \ref{oracc} shows that condition $p = \Omega\left(\sqrt{\frac{\log{n_1}}{n_1n_2}}\right)$ is necessary for the supervised oracle to achieve exact recovery when $n_2 \geq n_1\log{n_1}$. Combining this result with the sufficient conditions for exact recovery from Theorem \ref{th:th3}, {we can  now %the full picture in the zone where $ n_1(\log n_1)^4 = O(n_2)$. The diagram below 
complete the diagram in \cite{florescu2016spectral}, which compares {the exact recovery conditions for} SVD and for the debiased spectral method when $n_2 \geq n_1(\log n_1)^4$.} We recall here that the SVD estimator is the one returning signs of the second eigenvector of $AA^\top$. In \cite{florescu2016spectral}, a debiased spectral method is also considered, which uses as an estimator the signs of the second eigenvector of $AA^\top - \mathbb{E}(WW^\top)$. %the hollowed version $H(AA^\top)$. 
Under perfect balance (that is, $\gamma_1=\gamma_2=0$),
$\mathbb{E}(WW^\top)$ is proportional to $\mathbf{I}_{n_1}$ and hence SVD and debiased spectral method coincide in that case, while in general the debiased spectral method outperforms the SVD estimator. Comparison of {the oracle and of} the three methods: SVD, debiased spectral (DS) and hollowed Lloyd's (HL), in the general case of imbalance, {and under the condition $n_2 \geq n_1(\log n_1)^4$}, can be summarized as follows :

\begin{tikzpicture}

% horizontal axis
\draw[->] (0,0) -- (12,0) node[anchor=north] {$p^2$};
% labels
\draw	(0,0) node[anchor=north] {0}
		(2.7,0) node[anchor=north] {$\frac{\log{n_1}}{n_1n_2}$}
		(6,0) node[anchor=north] {$\frac{1}{n_1^{4/3}n_2^{2/3}}$}
		(9.3,0) node[anchor=north] {$\frac{1}{n_1^2}$};
% ranges
\draw	(1.5,1) node{\bf{\color{red}failure of }};
\draw	(1.5,0.5) node{\bf{\color{red} the oracle}};

\draw	(4.5,1) node{\bf{\color{red}failure of DS }};
\draw	(4.5,0.5) node{\bf{\color{blue}success of HL}};

\draw	(7.8,1) node{\bf{\color{red}failure of SVD }};
\draw	(7.8,0.5) node{\bf{\color{blue} success of DS}};

\draw	(10.5,1) node{\bf{\color{blue}success   }};
\draw	(10.5,0.5) node{\bf{\color{blue} of SVD }};

% vertical axis
%\draw[->] (0,0) -- (0,4) node[anchor=east] {$\Phi_{o}(a)$};
% nominal speed
\draw[dotted] (2.7,0) -- (2.7,2);
\draw[dotted] (9.3,0) -- (9.3,2);
\draw[dotted] (6,0) -- (6,2);

% Us
%\draw[thick] (0,3) -- (4,3);
%\draw[thick] (6,1) -- (10,1);%label

% Psis
%\draw[thick,dashed]  (4,3) parabola[bend at end] (6,1); %label

\end{tikzpicture}
 {This hierarchy of procedures becomes apparent in} the simulations given in the next section.

\section{Numerical experiments}\label{sec:numerical}

The goal of this section is to provide numerical evidence to our theory.  We compare the performance of methods defined previously, namely:
\begin{itemize}
\item SVD estimator (SVD),
    \item debiased spectral estimator (DS),
    \item diagonal deletion SVD estimator (DD),
    \item hollowed Lloyd's  algorithm with spectral initialization (HL),
    \item the oracle procedure (O).
\end{itemize}
 In what follows, we fix the number of labels  $n=300$,  the imbalance $\gamma_1 = 0$, $\gamma_2 = 0.5$ and $\delta = 0.5$. For the sake of readability of plots, we define the parameters $a$ and $b $ such that
\[
p = \sqrt{a} / n_1 \quad \text{and} \quad b = n_1(\log{n_1} )/ n_2.
\]
According to our improved sufficient conditions and using the above parameterization we expect the phase transition for exact recovery to happen at 
\[
a \geq C_\delta (b \vee b^2)
\]
for some $C_{\delta}>0$. %\simo{\sout{when $b \geq 1$}}. 
We set up the simulations as follows. We consider $b\in\{0.1,0.5,5\}$ and we take $a$ on a uniform grid of $20$ points in a region where the phase transition occurs. %\simo{\sout{Although our theory holds for $b\leq 1$, we added simulations for the case of $b >1$}}.
For each such $(a,b)$, we repeat the simulation $1000$ times. Figure 1 presents the empirical probabilities of exactly recovering the vector of true labels $\eta_1$. 
%(figure \ref{fig:var})(figure \ref{fig:max})

\begin{figure}[ht]
\centering
\hspace{-.cm}
\begin{subfigure}{}
  \centering
  \includegraphics[width=.31\linewidth]{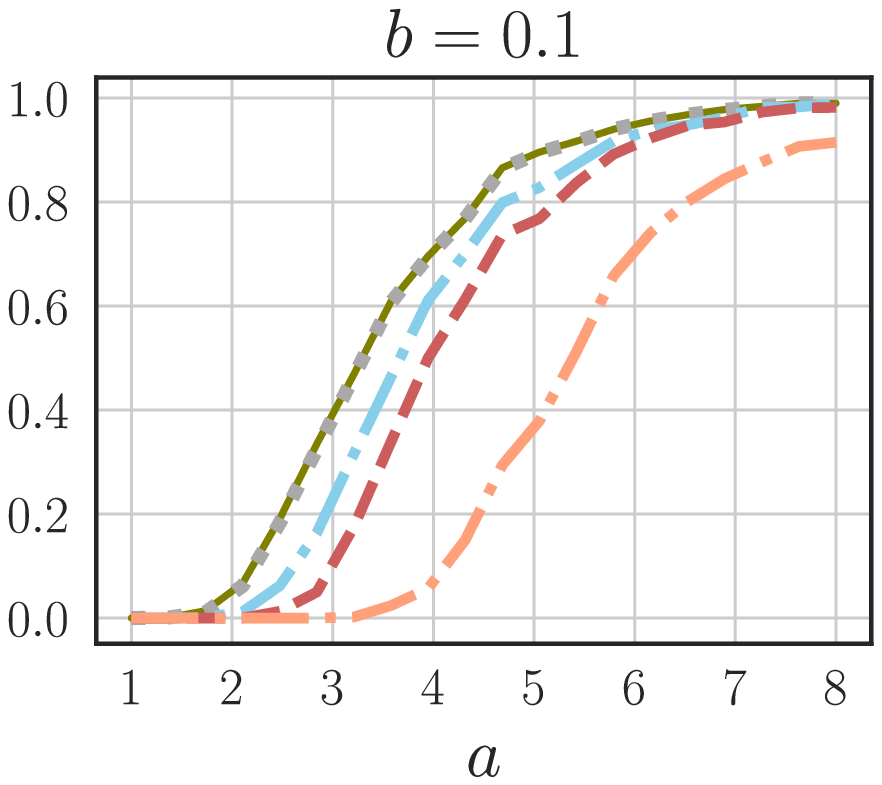}
  \label{fig:sub1}
\end{subfigure}
\centering
\begin{subfigure}
  \centering
  \includegraphics[width=.31\linewidth]{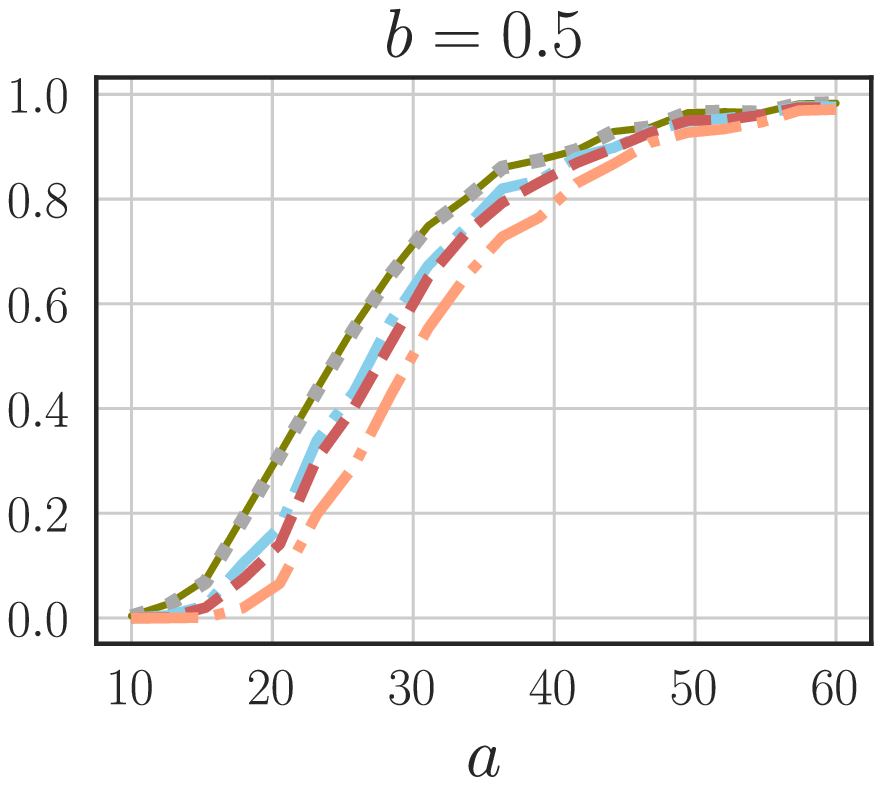}
  \label{fig:sub2}
\end{subfigure}
\centering
\begin{subfigure}
  \centering
  \includegraphics[width=.31\linewidth]{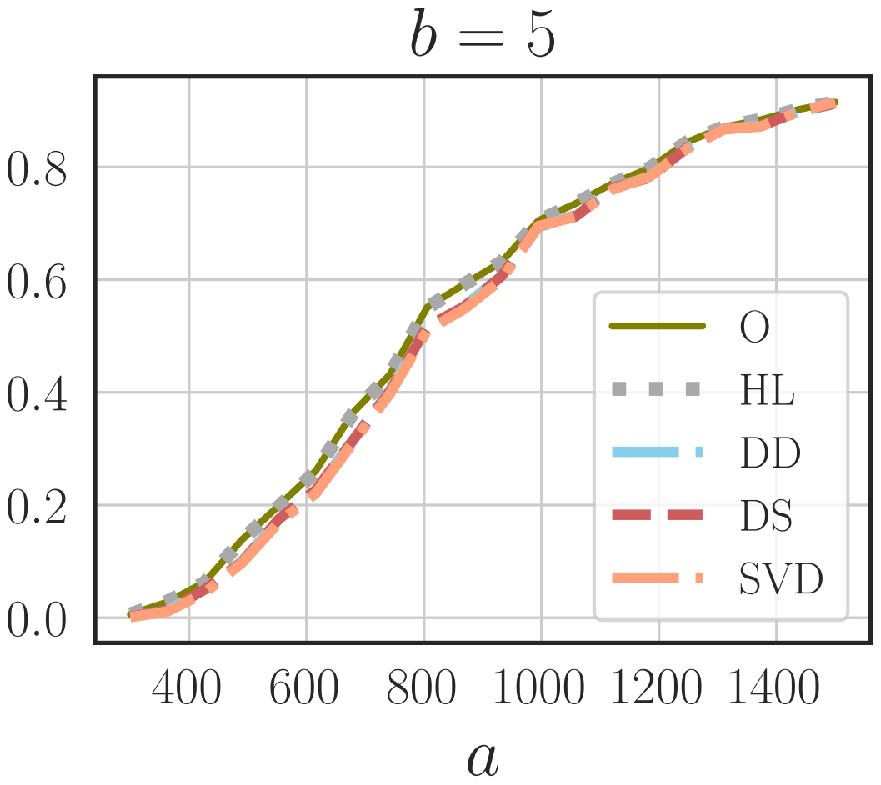}
  \label{fig:sub3}
\end{subfigure}

\caption{Empirical probability of success over 1000 runs of the experiment for: $b=0.1$ (left), $b=0.5$ (center) and $b=5$ (right).}
\label{fig:comparison}
\end{figure}
Overall, numerical experiments match our theoretical findings  and provide some interesting insights: 
\begin{enumerate}
    \item Hollowed Lloyd's algorithm with spectral initialization achieves a performance remarkably close to the oracle without any prior knowledge about the true labels. Notice that this holds also when only a fraction of labels can be recovered, i.e. when the probability of wrong recovery is not exactly zero. This, in particular, suggests that the theoretical comparison we established between the above algorithms can be extended beyond the problem of exact recovery. Further simulations  show that randomly initialized hollowed Lloyd's algorithm achieves the same performace as well (we omit these simulations since such an algorithm is not covered by our theory).
    \item In the case $b=0.1$ (high dimension), we recover empirically the diagram of Section \ref{sec:diagram}. Observe that as $b$ gets larger (moderate and small dimension) all the estimators converge to almost indistinguishable performance. {In other words,  the ranking of estimators given in Section \ref{sec:diagram}  only accentuates for the high-dimensional regime. This agrees with the fact that the conclusions of Section \ref{sec:diagram} are restricted to the zone $n_2 \geq n_1(\log n_1)^4$.}
    \item In high dimensions, the DD method outperforms the DS,  which supports the argument that, under the BSBM model, hollowing is more beneficial than debiasing (cf. Proposition \ref{lem:suboptimal} and the corresponding discussion). 
\end{enumerate}

 \section*{Acknowledgements}
This work was supported by GENES and by the French National Research Agency (ANR) under the grant Labex Ecodec (ANR-11-LABEX-0047). M. Ndaoud was partially supported by a James H. Zumberge Faculty Research and Innovation Fund at the University of Southern California and by the National Science Foundation grant CCF-1908905.

\bibliographystyle{imsart-nameyear}
%\bibliography{bibli}

%\newpage
\appendix
%\newpage
\section{Control of the spectral norm of the hollowed Gram matrix}

This section is devoted to the control of the spectral norm of the hollowed matrix $H(WW^{\top})=\sum_{j=1}^{n_2} H(W_jW_j^{\top})$,
where we denote by $W_j$ the columns of $W$.
The following theorem will be used in the proofs.
\begin{thm}[Matrix Bernstein inequality -- adapted from  \cite{tropp2012user}, Theorem 6.2]
\label{th:th4}
Let $(Y_j)_{j=1}^{n}$ be a sequence of independent symmetric random matrices of size $d \times d$, and $a,R>0$. Assume that for all $j$ in $\{ 1,\dots,n\}$ we have
\[
\mathbb{E} (Y_{j})=0 \text { and }\|\mathbb{E}(Y_j^q)\|_{\infty}  \le \frac{q!}{2}R^{q-2}a^2 \text { for $q=2,3,\dots$ } .
\]
Then, for all $t \geq 0$, 
\begin{align*}
\mathbb{P}\left( \Big\| \sum_{j=1}^{n} Y_{j}\Big\|_{\infty} \geq t\right) \leq d \, \exp \left(-\frac{t^{2}}{2 \sigma^{2}+2 R t}\right) \text { with } \sigma^{2}=na^2.% \left\|\sum_{j=1}^{n} A_j^{2}\right\|_{\infty}.
\end{align*}
\end{thm}
We will show that in our case this theorem can be applied with $Y_{j}=H(W_jW_j^{\top})$, $d=n_1$,  $n=n_2$, {$R = 3(1 + 2n_1p)$} and {$a^2 = 4 p^2n_1$}. One can check that it gives a strict improvement over the matrix Hoeffding type inequality that uses only the fact that $\|H(W_jW_j^\top) \|_{\infty} \leq n_1$ almost surely. Namely, we have the following theorem.
\begin{thm}\label{thm:new_bernstein}
For all $t \geq 0$, 
\begin{align*}
\mathbb{P}\left( \Big\| \sum_{j=1}^{n_2} H(W_jW_j^{\top})\Big\|_{\infty} \geq t\right) \leq n_1 \, \exp \left(-\frac{t^{2}}{8 n_1n_2p^2+6 (1+2n_1p) t}\right).
\end{align*}
\end{thm}

\begin{proof}
Fix $j$ in $\{ 1,\dots,n_2\}$. In view of Theorem \ref{th:th4}, it is enough to show that for all integers $q \geq 2$ we have
\begin{equation}\label{eq:moment_b}
 \left\| \mathbb{E}(H(W_jW_j^\top)^q)\right\|_{\infty} \leq 2q!(3(1+2n_1p))^{q-2}p^2n_1.
\end{equation}
We now prove \eqref{eq:moment_b}. To alleviate the notation, we set $w=W_j$ and we denote by $w_k$ the entries of $w$. Note that $w_k$ are independent random variables taking value $1-{\sf p}$ w.p. ${\sf p}$ and $-{\sf p}$ w.p. $1-{\sf p}$, where ${\sf p}$ is either $\delta p$ or $(2-\delta)p$. We have $\mathbb{E}(w_k)=0$ for all $k$. 
Furthermore, for any integer $m\ge 2$,
\begin{equation}\label{eq:w}
 |\mathbb{E}(w_k^m)|\le 2p.
\end{equation}
Indeed,
$$
 |\mathbb{E}(w_k^m)|\le {\sf p}(1-{\sf p}) \max_{0\le {\sf p}\le 1}((1-{\sf p})^{m-1}+ {\sf p}^{m-1})={\sf p}(1-{\sf p}).
$$
Denote by $h_{ik}(q)$ the $(i,k)$th entry of matrix $H(ww^\top)^q$. 
Note that the $(i,k)$th entry of matrix $H(ww^\top)$ is $H(ww^\top)_{ik} = w_iw_k\mathds{1}(i \neq k)$.  It comes out that 
\[
h_{ik}(q) =\sum_{(i_2, i_3,\dots , i_{q})\in J } w_iw_k \prod_{\ell=2}^{q}w_{i_{\ell}}^2,
\]
where $J=\{(i_2, i_3,\dots , i_{q}):  i_2\ne i_3, \dots, i_{q-1}\ne i_q; i_2\ne i, i_{q} \ne k \}$ and indices $i_\ell$ take values in $\{1,\dots,n_1\}$. Thus,
\begin{equation}\label{eq:hik}
|\mathbb{E}(h_{ik}(q))| \leq  \sum_{(i_2, i_3,\dots , i_{q})\in J} \bigg|\mathbb{E}\Big(w_iw_k\prod_{\ell=2}^{q}w_{i_{\ell}}^2\Big)\bigg|.
\end{equation}
First note that for $q=2$ the terms in this sum are non-zero only if $i=k$ and in this case the sum is bounded by $4p^2n_1$. Thus, \eqref{eq:moment_b} holds for $q=2$. In order to prove \eqref{eq:moment_b} for $q\geq 3$, it suffices to show that for all $i,k$  we have
\begin{equation}\label{eq:moment_b1}
|\mathbb{E}(h_{ik}(q)) |\leq  2q!(1+2pn_1)^{q-2}p^2, \quad i \neq k,
\end{equation}
and 
\begin{equation}\label{eq:moment_b2}
|\mathbb{E}(h_{ii}(q))| \leq  4q!(1+2pn_1)^{q-2}p^2n_1.
\end{equation}

We start by showing \eqref{eq:moment_b1} for all $q \geq 3$. Let $i\ne k$. We first bound the number of non-zero terms in the sum in \eqref{eq:hik}. Since $w_1,\dots,w_{n_1}$ are independent zero-mean random variables, the term in this sum corresponding to some fixed $(i_2, i_3,\dots , i_{q})$ can be non-zero only if both $i$ and $k$ belong to the set $\{i_2, i_3,\dots , i_{q}\}$.  In order to take into account equalities between different indices $i_\ell$, consider all partitions $\pi$ of the set $\{i_2, i_3,\dots , i_{q}\}$ into $s$ subsets, with equal indices in each subset, where $s$ runs from 2 to $q-1$ (the case $s=1$, that is $i_2= i_3 =\cdots = i_{q}$,  is excluded since the corresponding expectation vanishes). 

Assume a partition $\pi$ in $s$ subsets fixed. Then, for the expectation
$$
\mathbb{E}\bigg(w_iw_k\prod_{\ell=2}^{q}w_{i_{\ell}}^2\bigg)
$$
 to be non-zero, two out of $s$ subsets must contain variables with indices $i$ and $k$, and in this case due to independence of $w_m$ and \eqref{eq:w} we have
 \begin{equation}\label{eq:ps}
 \bigg|\mathbb{E}\Big(w_iw_k\prod_{\ell=2}^{q}w_{i_{\ell}}^2\Big)\bigg|\le (2p)^s.
\end{equation}
Denote by $\mathcal{P}_{s,2}$ the set of all partitions $\pi$ of $\{i_2, i_3,\dots , i_{q}\}$ into $s$ subsets such that for two of these subsets the indices $i_\ell$ are equal to $i$ and $k$. To get an upper bound on the cardinality of $\mathcal{P}_{s,2}$, notice that any such partition can be obtained by choosing $s-2$ distinct indices among the $q-3$ possible values (other than $i$ and $k$) and then allocating the remaining $q-s$ indices to $s$ buckets. 
This leads to the bound ${\rm Card}(\mathcal{P}_{s,2}) \le \binom{q-3}{s-2}s^{q-s}$.  Denote by $i_1(\pi)\ne \cdots \ne i_{s-2}(\pi)$ the $s-2$ distinct indices (other than $i$ and $k$) corresponding to the 
 partition $\pi\in \mathcal{P}_{s,2}$.  Using \eqref{eq:ps} and the fact that the indices $i_\ell(\pi)$ can take values from 1 to $n_1$ we obtain
\begin{align*}
|\mathbb{E}(h_{ik}(q))| &\leq \sum_{s=2}^{q-1} \sum_{\pi\in \mathcal{P}_{s,2}} \,\, \sum_{i_1(\pi) \neq  \cdots \neq i_{s-2}(\pi)} (2p)^s 
\\
&\leq \sum_{s=2}^{q-1}\binom{q-3}{s-2}s^{q-s} n_1^{s-2} (2p)^s\\
&\leq 2p^2q!\sum_{s=2}^{q-1}\binom{q-2}{s-2} (2pn_1)^{s-2}\\
&\leq 2q!(1+2pn_1)^{q-2}p^2,
\end{align*}
where we have used the inequalities $s^{q-s} \le (q-1)!/(s-1)!\le q!/2$. Thus, the bound \eqref{eq:moment_b1} is proved for all $q\ge 3$.

It remains to show that \eqref{eq:moment_b2} holds for $q\ge 3$. Denote by $\mathcal{P}_{s,1}$ the set of all partitions $\pi$ of $\{i_2, i_3,\dots , i_{q}\}$ into $s$ subsets such that for one of these subsets the index $i_\ell$ is equal to $i$. Similarly to the argument for $\mathcal{P}_{s,2}$, we obtain that
${\rm Card}(\mathcal{P}_{s,1}) \le \binom{q-2}{s-1}s^{q-s}$ and 
\begin{align*}
|\mathbb{E}(h_{ik}(q))| &\leq \sum_{s=2}^{q-1} \sum_{\pi\in \mathcal{P}_{s,1}} \,\, \sum_{i_1(\pi) \neq  \cdots \neq i_{s-1}(\pi)} (2p)^s 
\\
&\leq \sum_{s=2}^{q-1}\binom{q-2}{s-1}s^{q-s} n_1^{s-1} (2p)^s\\
&\leq \sum_{s=2}^{q-1}\binom{q-2}{s-2} \frac{q-s}{s-1} \frac{(q-1)!}{(s-1)!} n_1^{s-1} (2p)^s\\
&\leq 4p^2n_1 q!\sum_{s=2}^{q-1}\binom{q-2}{s-2} (2pn_1)^{s-2}\\
&\leq 4q!(1+2pn_1)^{q-2}p^2n_1.
\end{align*}
%Hence we conclude that, for all $t \geq 0$, 
%\begin{align*}
%\mathbb{P}\left( \left\| \sum_{j=1}^{n_2} H(W_jW_j^{\top})\right\|_{\infty} \geq t\right) \leq n_1 \, \exp \left(-\frac{t^{2}}{8 n_1n_2p^2+6 (1+2n_1p) t}\right).
%\end{align*}
\end{proof}
%\simo{
 \secondlemma*
\begin{proof}
Introduce the notation ${\sf H}=\|H(WW^\top)\|^2_{\infty}$, $t_1=6\sqrt{n_1n_2p^2\log n_1}$, $t_2= 48(1+2n_1p)\log{n_1}$ and $t_3=t_1 \vee t_2$. Using Theorem \ref{thm:new_bernstein} {and the facts that $\exp(-a/(b+c))\le \exp(-a/(2b)) + \exp(-a/(2c))$ for all $a,b,c>0$,} we get  
\begin{align*}
\mathbb{E}\left({\sf H}^2\right) &=2\int_0^{\infty} {\mathbb P}({\sf H}>t)\,t dt \le t_3^2 + 2 n_1\int_{t_3}^{\infty}\exp \left(-\frac{t^{2}}{{16} n_1n_2p^2}\right) t dt 
\\
&\quad +  2 n_1\int_{t_3}^{\infty}\exp \left(-\frac{t}{{12}(1+ 2n_1 p)}\right) t dt.
\end{align*}
Since $n_1n_2 p^{2}\geq C\log{n_1}$ it comes out that, for some constants $c_1, c_2, c_3$ we have
\begin{align*}
\mathbb{E}\left({\sf H}^2\right) &\leq c_1 t_3^2\\
&\leq c_2 (n_1n_2p^2\log{n_1} + n_1^2p^2\log^2{n_1})\\
&\leq c_3 n_1n_2p^2\log{n_1}(1 + n_1\log{n_1}/n_2).
\end{align*}
%{\color{blue}Sasha: in this theorem it is enough to assume $n_1n_2 p^{2}\geq C\log{n_1}$.} \simo{True but it is not a problem I think.}
\end{proof}
%}
 \firstlemma*
%Let $n_1,n_2 \geq 2$ such that $n_2\geq n_1\log{n_1}$. Assume that $\gamma_1\gamma_2<1$ and $ n_1^{-1}n_2^{-1}\log{n_1} \leq p^2 \leq (n_1\log{n_1})^{-2}/4115^2$. Then
%\[
%\mathbb{E}\left(\|WW^{\top} - \mathbb{E}(WW^{\top})\|^2_{\infty} \right) \geq \frac{n_2 p }{20}.
%\]

 \begin{proof}
 Set ${\sf p}=\max(\delta p, (2-\delta)p)\ge p$. Since $\gamma_1<1$ and $\gamma_2<1$ then at least one row of $W$ has not less than $n_2/2$ entries that are centered Bernoulli variables with parameter ${\sf p}$. Without loss of generality, let it be the first row of $W$. We denote this first row by $X_1$. 
We have
\begin{align*}
\|WW^{\top} - \mathbb{E}(WW^{\top})\|_{\infty} &\geq \| WW^{\top} - \mathbb{E}(WW^{\top})\|_{\infty} - \|H(WW^{\top})\|_{\infty}\\
& \geq |\|X_{1}\|^2 - \mathbb{E}(\|X_{1}\|^2)|  - \|H(WW^{\top})\|_{\infty},
\end{align*}
so that
\[
\mathbb{E}\left(\|WW^{\top} - \mathbb{E}(WW^{\top})\|^2_{\infty} \right) \geq \frac{1}{2}\mathbb{E}\left(( \|X_{1}\|^2 - \mathbb{E}(\|X_{1}\|^2))^2\right) - \mathbb{E}\left(\|H(WW^\top)\|^2_{\infty}\right).
\]
Denoting by $\eta$ the centered Bernoulli variable with parameter ${\sf p}$ ($\eta$ takes value $1-{\sf p}$ with probability ${\sf p}$ and value $-{\sf p}$ with probability $1-{\sf p}$) we get
\begin{align*}
\mathbb{E}\left(( \|X_{1}\|^2 - \mathbb{E}(\|X_{1}\|^2))^2\right) & \ge \frac{n_2}{2} {\rm Var}(\eta^2)  
\\
&= \frac{n_2}{2} {\sf p}(1-{\sf p})(1-2{\sf p})^2 \geq \frac{9n_2p}{20},
\end{align*}
where we have used the inequalities $p\le {\sf p} \leq 2p \le 1/70$. 

Next, note that $2n_1p\le 1$ and introduce again the notation ${\sf H}=\|H(WW^\top)\|^2_{\infty}$, $t_1=4\sqrt{n_1n_2p^2\log n_1}$, $t_2= 4n_1n_2p^2/ (3(1+2n_1 p))>t_1$. From Theorem \ref{thm:new_bernstein} and the facts that $t_2\ge (2/3)n_1n_2p^2\ge (2c/3)\log{n_1} $ with $c=18^2$, and $n_1\ge2$ we get 
\begin{align*}
\mathbb{E}\left({\sf H}^2\right) &=2\int_0^{\infty} {\mathbb P}({\sf H}>t)\,t dt \le t_1^2 + 2 n_1\int_{t_1}^{t_2}\exp \left(-\frac{t^{2}}{16 n_1n_2p^2}\right) t dt 
\\
&\quad +  2 n_1\int_{t_2}^{\infty}\exp \left(-\frac{t}{12(1+ 2n_1 p)}\right) t dt
\\
&\leq 16n_1n_2p^2\log{n_1}+ 16n_1n_2p^2 + 2 n_1\int_{t_2}^{\infty}\exp \Big(-\frac{t}{24}\Big) t dt
%\\
%&
%= 16n_1n_2p^2(\log{n_1}+1) +  2 n_1(24 t_2 +(24)^2)\exp(-{t_2}/{24})
\\
&
\le 16n_1n_2p^2(\log{n_1}+1) +  2 n_1(32 n_1n_2p^2 +(24)^2)\exp(-c(\log n_1)/{36})
%\\
%&
%= 16n_1n_2p^2(\log{n_1}+1) +  2 n_1^{-9}(32 n_1n_2p^2 +(24)^2)
\\
&
\le 16n_1n_2p^2(\log{n_1}+1) +  2^{-3} n_1n_2p^2 +(3/2)^2
\\
&\leq n_1n_2p^2\log{n_1}\Big(16 + \frac{17}{\log{n_1}}\Big)
\\
& \leq \frac{n_2p}{5},
\end{align*}
where we have used the condition on $p$. Combining the above displays we get the proposition.
\end{proof}

\vspace{3mm}
 
\section{Lower bound on the oracle}

We define the oracle as follows 
\[
\tilde\eta_1 = \sign \left( H((A-p\mathds{1}_{n_1}\mathds{1}_{n_2}^\top)(A-p\mathds{1}_{n_1}\mathds{1}_{n_2}^\top)^\top)\eta_1\right).
\]
\firstprop*
%\begin{prop} Assume that $n_2 \geq n_1\log{n_1}$ and $n_{1+} = n_{1-}$, $n_{2+} = n_{2-}$. There exists $c_\delta>0$ depending only on %$\delta$ such that if $p^2 = c_\delta\frac{\log{n_1}}{n_1n_2}$ then
%\begin{align*}
%\underset{n_1 \to \infty}{\lim}  \sum_{i=1}^{n_1}\mathbb{P}(\tilde\eta_{1i} \neq \eta_{1i}) =\infty.
%\end{align*}
%\[
%\underset{n_1 \to \infty}{\lim} \mathbb{E}|\eta^{o} - \eta| = \infty.
%\]
%\end{prop}
\begin{proof}
Since $n_{1+} = n_{1-}$ and $n_{2+} = n_{2-}$ we obtain that the $i$th entry of vector $H((A-p\mathds{1}_{n_1}\mathds{1}_{n_2}^\top)(A-p\mathds{1}_{n_1}\mathds{1}_{n_2}^\top)^\top)\eta_{1} $ is equal to
$$
h_i = \sum_{j =1}^{n_2}(A_{ij}-p)\sum_{k =1}^{n_1}A_{kj}  \eta_{1k} -  \eta_{1i}\sum_{j =1}^{n_2}(A_{ij}-p)^2.
$$
For all $i$ in $\{1,\dots,n_1\}$, since $\tilde \eta_{1i} \neq \eta_{1i}$ is equivalent to $h_i\eta_{1i}<0$ we have
\begin{align*}
\mathbb{P}(\tilde \eta_{1i} \neq \eta_{1i})=\mathbb{P}\left( \sum_{k \neq i} \sum_{j =1}^{n_2} \eta_{1i}\eta_{1k}(A_{ij}-p)A_{kj}  < p \sum_{j =1}^{n_2} (p-A_{ij}) \right). 
\end{align*}
Observe that
\begin{align*}
    \sum_{k \neq i} \sum_{j =1}^{n_2}\eta_{1i}\eta_{1k}(A_{ij}-p)A_{kj} &= (1-p)\sum_{k \neq i, \ j:A_{ij}=1} \eta_{1i}\eta_{1k}A_{kj} - p\sum_{k \neq i, \ j:A_{ij}=0}\eta_{1i}\eta_{1k}A_{kj}\\
    &= -(1-p)\sum_{k \neq i:\eta_{1k}\neq \eta_{1i}} \ \sum_{j:A_{ij}=1} A_{kj}\\
    & +(1-p)\sum_{k \neq i:\eta_{1k}= \eta_{1i}} \ \sum_{j:A_{ij}=1} A_{kj} - p\sum_{k \neq i,\ j:A_{ij}=0}\eta_{1i}\eta_{1k}A_{kj}.
\end{align*}
Hence
\begin{align*}
\mathbb{P}(\tilde \eta_{1i} \neq \eta_{1i}) &= \mathbb{P}\left( \sum_{k \neq i:\eta_{1k}\neq \eta_{1i}} \ \sum_{j:A_{ij}=1}  A_{kj} > \beta  \right)
\geq \mathbb{P}\left( \alpha > \beta  \right)
\\
&\geq \mathbb{E}\big[\mathbb{P}( \alpha > \beta \big\vert A_i )\mathds{1}_{F}\big]
,
\end{align*}
where $A_i=(A_{ij})_{j=1}^{n_2}$, 
\[
\alpha = \sum_{k \neq i:\eta_{1k}\neq \eta_{1i}}  \ \sum_{j:A_{ij}=1,\eta_{2j}=\eta_{1i}} A_{kj},
\]
\begin{align*}
\beta
&=\sum_{k \neq i:\eta_{1k}= \eta_{1i}} \ \sum_{j:A_{ij}=1}A_{kj} +   \frac{p}{1-p}\bigg(\sum_{j =1}^{n_2} A_{ij}-\sum_{k \neq i, \ j:A_{ij}=0}\eta_{1i}\eta_{1k}A_{kj}    \bigg),
\end{align*}
and
\[
F = \Big\{ \sum_{j =1}^{n_2} A_{ij}  \leq 4n_2p \Big\}\cap\Big\{ \sum_{j: \eta_{2j} = \eta_{1i}} A_{ij} \geq \delta pn_2/4 \Big\}.
\]
Note that  $F$ is an event of large enough probability for $n_1$ large enough. Indeed, as $\mathbb{E}(A_{ij})\le 2p$ and ${\rm Var}(A_{ij})\le 2p$ we get from Chebyshev inequality that
\begin{align}\label{eq:markov}
\mathbb{P}\Big( \sum_{j =1}^{n_2} A_{ij}  > 4n_2p \Big)&\le \mathbb{P}\Big( \sum_{j =1}^{n_2} (A_{ij}-\mathbb{E}(A_{ij}))  > 2n_2p \Big)
\\
&\le \frac{1}{2p n_2 }\le \frac{1}{2\sqrt{c_\delta} \log n_1},\nonumber
\end{align} 
where we have used the fact that $n_2\ge n_1\log n_1$.
Similarly, using Chebyshev inequality and the facts that for any $i$ we have ${\rm Card}\{j: \eta_{2j} = \eta_{1i}\}= n_2/2$ and that $\mathbb{E}(A_{ij})= \delta p$ for $\eta_{2j}=\eta_{1i}$ we find
\begin{align}\label{eq:cheb}
\mathbb{P}\Big( \sum_{j: \eta_{2j} = \eta_{1i}} A_{ij} < \delta pn_2/4 \Big) &\le \mathbb{P}\Big(\sum_{j: \eta_{2j} = \eta_{1i}} (\delta p - A_{ij} )> \delta pn_2/4 \Big) \\
& \le \frac{8}{\delta pn_2} \le  \frac{8}{\delta\sqrt{c_\delta} \log n_1}.\nonumber
\end{align} 
It follows from \eqref{eq:markov} and \eqref{eq:cheb} that
\begin{align}\label{eq:fff} 
\mathbb{P}(F)\ge 1 - \frac{1}{\sqrt{c_\delta} \log n_1}\Big(\frac12 + \frac{8}{\delta}\Big).
\end{align} 
Next, from Chebyshev inequality and the facts that $\mathbb{E}(A_{kj})\le 2p$, ${\rm Var}(A_{kj})\le 2p$, we obtain, conditionally on $A_i$,
\begin{align*}
\mathbb{P}\Big(\Big|\sum_{k \neq i, \ j:A_{ij}=0}\eta_{1i}\eta_{1k}A_{kj} \Big| \geq  4n_2n_1p \big| \,A_i\Big) 
& \leq \frac{1}{2n_2n_1p }\,.
\end{align*} 
Quite similarly, as ${\rm Card}\{k: \eta_{1k} = \eta_{1i}\}= n_1/2$ and for $A_i\in F$ we have ${\rm Card}\{j: A_{ij}=1\}= \sum_{j =1}^{n_2} A_{ij} \le 4p n_2$, the following inequality holds
 \begin{align*}
\forall \, A_i\in F: \quad \mathbb{P}\left( \sum_{k \neq i : \eta_{1k} = \eta_{1i}} \  \sum_{ j:A_{ij}=1} A_{kj}  \geq  8n_2n_1p^2\Big| \, A_i\right) \leq \frac{1}{4n_2n_1p^2 }\,.
\end{align*} 
Thus, for all $n_1$ large enough to have $p \leq 1/2$ we obtain
\begin{align*}
\forall \, A_i\in F: \quad \mathbb{P}( \beta  \leq 24 c_{\delta}\log n_1 | \, A_i ) &=\mathbb{P}( \beta  \leq 24 n_1n_2p^2 | \, A_i ) 
\\
&\geq 1 - \frac{3}{4n_2n_1p^2 }= 1 - \frac{3}{4c_{\delta}\log n_1}.
\end{align*}
Observe that random variables $\alpha$ and $\beta$ are independent conditionally on $A_i$ since the sums over $(k,j)$ in their definitions are taken over disjoint sets of indices. Using this we get
\begin{align}\label{eq:alpha}
\mathbb{P}(\tilde \eta_{1i} \neq \eta_{1i}) &\geq \Big(1 - \frac{3}{4c_{\delta}\log n_1}\Big)\mathbb{E}\left[\mathbb{P}\left( \alpha \ge 24 c_{\delta}\log n_1\Big\vert A_i \right)\mathds{1}_{F}\right].
%\\
%&\geq \frac{c'}{4\sqrt{5n_1n_2p^2}}\exp\left( - 5n_1n_2p^2 \log(50/\delta) \right)\mathbb{P}(F),
\end{align}
Note that, conditionally on $A_i$, the random variable $\alpha$ has a Binomial distribution with probability parameter $(2-\delta) p$.  Moreover, if  $A_i\in F$ then the number of terms in $\alpha$ denoted by $n$  is such that $n \le 4pn_1n_2$ and $n\ge (n_1/2-1)(\delta p n_2/4)\ge \delta p n_1 n_2/12$ for $n_1\ge 6$. It follows that, for any fixed $A_i\in F$, the assumptions of  
Lemma \ref{lem:lower_bound} are satisfied with ${\sf p}=(2-\delta)p$, $t=24 c_{\delta}\log n_1$ provided that $\sqrt{n_1/\log n_1}> 288\sqrt{c_\delta}/\delta$. Therefore, for $n_1$ large enough to satisfy this condition and $c_{\delta}\log n_1\ge 1$, $n_1\ge 6$, Lemma \ref{lem:lower_bound} implies that, for any $ A_i\in F$,
\begin{align*}
\mathbb{P}( \alpha \ge 24 c_{\delta}\log n_1\vert A_i)&\geq \frac{e^{-1/6}}{\sqrt{50 \pi c_{\delta}\log n_1}}\exp\left( - 25c_{\delta}\log n_1 \log\Big(\frac{300}{\delta(2-\delta) }\Big) \right).
\end{align*}
With the choice $c_{\delta}= \Big(50 \log\Big(\frac{300}{\delta(2-\delta) }\Big)\Big)^{-1}$ this yields 
\begin{align*}
\mathbb{P}( \alpha \ge 24 c_{\delta}\log n_1 \vert A_i)&\geq \frac{e^{-1/6}}{\sqrt{50 \pi c_{\delta}n_1\log n_1}}.
\end{align*}
Combining this inequality with \eqref{eq:fff} and \eqref{eq:alpha} we get the proposition.
\end{proof}

The following lemma is used to control the lower tail of binomial variables. 

\begin{lem}%[adapted from Theorem 2 in \cite{aratia}]
\label{lem:lower_bound}\hspace{-4mm}
Let $\xi_1,\dots,\xi_n$ be i.i.d. Bernoulli random variables with parameter ${\sf p}$ and $\alpha= \sum_{i=1}^n \xi_i $. Then for all $ n{\sf p} < t < n$ we have
\begin{equation*}
\mathbb{P}\left( \alpha \geq t \right)  \geq \frac{e^{-1/6}}{\sqrt{2\pi(t+1)}}\exp\left(-(t+1)\log\Big(\frac{t+1}{n{\sf p}} \Big)\right).
\end{equation*}
\end{lem}
\begin{proof} Set $k=\lceil t\rceil$. Since 
$$
\mathbb{P}\left( \sum_{i=1}^n \xi_i \geq t \right)  \geq \mathbb{P}\left( \sum_{i=1}^n \xi_i = k \right) 
$$
for $k=n$ the result is trivial. Assume that $k\le n-1$ and set $a=k/n$. Then ${\sf p}< a < 1$. By Stirling's approximation,
$$
\sqrt{2\pi n} \,(n/e)^{n}  \le n! \le \sqrt{2\pi n}\, (n/e)^{n}   e^{1/12}.
$$
Therefore,
\begin{align*}
\mathbb{P}\left( \sum_{i=1}^n \xi_i \geq t \right)  &\geq \mathbb{P}\left( \sum_{i=1}^n \xi_i = k \right)  = \frac{n! {\sf p}^k(1-{\sf p})^{n-k}}{k!(n-k)!}\\
&\ge
\frac{\sqrt{2\pi n} \,n^{n}  {\sf p}^k(1-{\sf p})^{n-k}}{e^{1/6} \sqrt{2\pi k} \,k^{k} \sqrt{2\pi (n-k)} \,(n-k)^{n-k}}\\
&\ge
\frac{ {\sf p}^k(1-{\sf p})^{n-k}}{e^{1/6} \sqrt{2\pi an} \,  \,a^{k}  \,(1-a)^{n-k}}\ge \frac{ {\sf p}^k}{e^{1/6} \sqrt{2\pi an} \,  \,a^{k} }.
\end{align*}
\end{proof}

\section{Main proofs}

\subsection{Proof of Theorem \ref{th:th1}}

Recall that 
\begin{align*}
\eta^0_1 = \sign(\hat{v}),
\end{align*}
where $\hat{v}$ is the eigenvector corresponding to the top eigenvalue of the matrix 
$$H((A-\hat{p}\mathds{1}_{n_1}\mathds{1}_{n_2}^\top )(A-\hat{p}\mathds{1}_{n_1}\mathds{1}_{n_2}^\top )^{\top})$$
with $\hat{p}= \frac{1}{n_1n_2}\mathds{1}^{\top}_{n_1}A\mathds{1}_{n_2}$.
%Assume $\gamma_1\gamma_2 \leq \sqrt{\alpha}/144$. In what follows we denote by 
Recall the notation ${\tilde A}= A - p\mathds{1}_{n_1}\mathds{1}^\top_{n_2}$. We have
\[
H((A-\hat{p}\mathds{1}_{n_1}\mathds{1}_{n_2}^\top )(A-\hat{p}\mathds{1}_{n_1}\mathds{1}_{n_2}^\top )^{\top}) = H({\tilde A}{\tilde A}^\top)+Z_4,
\]
where (cf. \eqref{eq:eq1}) 
$$
H({\tilde A}{\tilde A}^\top) = (\delta -1)^2 p^2 n_2 H(\eta_{1} \eta_{1}^\top) +H(W W^\top) + p(\delta-1)H( W\eta_{2}\eta_{1}^{\top} +  \eta_{1}\eta_{2}^{\top}W^\top) 
$$
and
\[
Z_4: =  H((A-\hat{p}\mathds{1}_{n_1}\mathds{1}_{n_2}^\top )(A-\hat{p}\mathds{1}_{n_1}\mathds{1}_{n_2}^\top )^{\top})  - H({\tilde A}{\tilde A}^{\top}).
\] 
Therefore,
\begin{align*}
H((A-\hat{p}\mathds{1}_{n_1}\mathds{1}_{n_2}^\top )(A-\hat{p}\mathds{1}_{n_1}\mathds{1}_{n_2}^\top )^{\top}) = (\delta-1)^2 p^2 n_2 \eta_1 \eta_1^\top + Z,
\end{align*}
where 
\begin{align*}
Z =  \underbrace{H(W W^\top)}_{Z_{1}} + \underbrace{p(\delta-1)H( W\eta_{2}\eta_{1}^{\top} +  \eta_{1}\eta_{2}^{\top}W^\top)}_{Z_{2}}  - \underbrace{(\delta-1)^2 p^2 n_2 I_{n_1}}_{Z_{3}} + Z_4.
\end{align*}
Notice that since $Z_3$ is a multiple of the identity matrix, $\hat{v}$ is the eigenvector corresponding to the top eigenvalue of $H' = H((A-\hat{p}\mathds{1}_{n_1}\mathds{1}_{n_2}^\top )(A-\hat{p}\mathds{1}_{n_1}\mathds{1}_{n_2}^\top )^{\top})+Z_3$.
Thus, $\hat{v}$ and $\frac{1}{\sqrt{n_1}}\eta_1$ are the eigenvectors of $\frac{1}{n_1} H'$ and $ (\delta-1)^2 p^2 n_2 \frac{ \eta_1 \eta_1^\top}{n_1}$ associated to their top eigenvalues, respectively. Since $\eta_1\eta_1^{\top}$ is rank one matrix, we get from Davis-Kahan Theorem (Theorem 4.5.5. in \cite{vershynin2018high}) that 
\begin{align*}
\min_{\nu \in \{-1, 1\}} \left \| \frac{1}{\sqrt{n_{1}}}\eta_1 - \nu \hat{v}  \right\|_2^2 \leq \frac{8 \big \| Z_1 +Z_2 +Z_4 \big \|_{\infty}^2}{(\delta-1)^{4}p^{4}n_{1}^{2}n_{2}^{2}}.
\end{align*}
This implies (see Lemma \ref{lem:surrogate} below) that %(see, e.g., Lemma $20$ in \cite{ndaoud2018sharp}) that 
\begin{align*}
    \frac{1}{n_1} r(\eta_1, \eta_1^0) \leq \frac{16}{(\delta-1)^{4}p^{4}n_{1}^{2}n_{2}^{2}} \|  Z_1 +Z_2 +Z_4 \|_{\infty}^2.
\end{align*}
Thus, in order to bound $r(\eta_1,\eta_1^0)$, it remains to control the spectral norm of $Z_1 +Z_2 +Z_4$. Namely, we will prove that 
$$
\underset{n_{1} \to \infty}{\lim} \mathbb{P}\left(\|Z_{i}\|_{\infty} \geq \frac{\sqrt{\alpha}}{12}(\delta-1)^{2}p^{2}n_{1}{n_{2}}\right) = 0, \quad i=1,2,4,
$$
which implies the theorem. 
\begin{itemize}

\item \textbf{Control of $\|Z_{1}\|_{\infty}$.}

\noindent Recall that $W$ is a random matrix with  entries that are independent and distributed as $\zeta - \mathbb{E}(\zeta)$ where $\zeta$ is a Bernoulli random variable with parameter $\delta p$ or $(2-\delta)p$. Therefore, both the expectation and the variance of each entry are bounded by $2p$. 
%Let $W_1, \dots, W_{n_2}$ be the columns of matrix $W$. 
We now apply Theorem~\ref{thm:new_bernstein} with $t=\frac{\sqrt{\alpha}}{12}(\delta-1)^{2}p^{2}n_{1}n_{2}$. This yields  
\begin{align*}
     \mathbb{P}\left( \|Z_{1}\| _{\infty} \geq  \frac{\sqrt{\alpha}}{12}(\delta-1)^{2}p^{2}n_{1}n_{2}\right) &
     \leq n_1 \exp \left[-\frac{t^{2}}{(8n_1n_2p^2+6t) + 2n_1p t}\right]\\
     & \leq  {n_1} \exp \left[-{12^{-2}\alpha (\delta-1)^4}n_1n_2p^2/17\right] \\
     &\quad + {n_{1}}\exp\left[-\sqrt{\alpha}(\delta-1)^2pn_2/{288}\right],
\end{align*}
where the last inequality uses the facts that $\exp(-a/(b+c))\le \exp(-a/(2b)) + \exp(-a/(2c))$ for all $a,b,c>0$, and $\alpha\in (0,1)$, $|\delta-1|<1$.
Recall that $p \geq C (\delta-1)^{-2}\sqrt{\frac{\log n_{1}}{n_{1}n_{2}}}$ and  $ p \geq C (\delta-1)^{-2}\frac{\log{n_1}}{n_2}$   by the assumption of the theorem.
Using these conditions and choosing $C\ge 289/\sqrt{\alpha}$ we obtain
$$
\mathbb{P}\left( \|Z_{1}\| _{\infty} \geq  \frac{\sqrt{\alpha}}{12}(\delta-1)^{2}p^{2}n_{1}n_{2}\right)   
 \le 2n_{1}^{-\frac{1}{288}}.
$$

\item \textbf{Control of $\|Z_{2}\|_{\infty}$.}

\noindent In order to control $Z_{2}$, we first observe, using the inequality $\|H(M)\|_\infty \le 2\|M\|_\infty$ valid for any matrix $M\in \mathbb{R}^{n_1\times n_1}$ (cf., e.g., Lemma 17 in \cite{ndaoud2018sharp}), that
\begin{align*}
\big \| H \left( \eta_{1}\eta_{2}^{\top}W^\top + W \eta_{2}\eta_{1}^\top \right) \big \|_{\infty} & \leq 2\big \| \eta_{1}\eta_{2}^{\top} W^\top + W \eta_{2}\eta_{1}^\top \big \|_{\infty} \\
& \leq 2 \big \| \eta_{1}\eta_{2}^{\top}W^\top \big \|_{\infty} + 2 \| W\eta_{2}\eta_{1}^\top \big \|_{\infty} \\
& \leq 4 \sqrt{n_{1}}\| W\eta_{2}\|_{2}.
\end{align*}
Hence 
\begin{equation}\label{Z2_1}
 \mathbb{E}(\|Z_{2}\|_{\infty}^{2}) \leq 16(\delta-1)^{2}p^{2}n_{1}\mathbb{E}(\|W\eta_{2}\|_{2}^{2}).   
\end{equation}
Denote by  $X_{1},\dots,X_{n_1}$ the column vectors equal to the transposed rows of matrix $W$.
 Since $\mathbb{E}(X_{i}X_{i}^{\top})$ is a diagonal matrix with positive entries bounded from above by $2p$ for all $i=1,\dots,n_{1}$, we obtain 
\begin{equation}\label{Z2_2}
\mathbb{E}(\|W\eta_{2}\|_{2}^{2}) = \eta_{2}^{\top} \mathbb{E}(W^{\top}W) \eta_{2}=  \sum_{i=1}^{n_{1}}\eta_{2}^{\top}\mathbb{E}(X_{i}X_{i}^{\top})\eta_{2} \leq 2pn_{1}n_{2}.
\end{equation}
Chebyshev's inequality combined with  \eqref{Z2_1} and \eqref{Z2_2} yields the bound
\begin{align*}
\mathbb{P}\left(\|Z_{2}\|_{\infty} \geq \frac{\sqrt{\alpha}}{12}(\delta-1)^{2}p^{2}n_{1}{n_{2}}\right) %&\leq \frac{9\cdot 2^{9}(\delta-1)^{2}p^{3}n_{1}^{2}n_{2}}{(\delta-1)^{4}\alpha p^{4}n_{1}^{2}n_{2}^{2}}
%\\ 
&\leq \frac{9\cdot 2^{9}}{\alpha (\delta-1)^{2}pn_{2}} \\
&\leq \frac{9\cdot 2^{9}}{C \alpha  \log n_{1}},
\end{align*}
where we have used the fact that
$p \geq C (\delta-1)^{-2}\frac{\log n_{1}}{n_{2}}$ by the assumptions of the theorem. \vspace{3mm}

\item \textbf{Control of $\|Z_{4}\|_{\infty}$.}

We have
\begin{align*}
    Z_4 &= H((A-\hat{p}\mathds{1}_{n_1}\mathds{1}_{n_2}^\top )(A-\hat{p}\mathds{1}_{n_1}\mathds{1}_{n_2}^\top )^{\top} - (A-p\mathds{1}_{n_1}\mathds{1}_{n_2}^\top )(A-p\mathds{1}_{n_1}\mathds{1}_{n_2}^\top )^{\top})\\
    & = H( (p-\hat{p})(A\mathds{1}_{n_2}\mathds{1}_{n_1}^\top+ \mathds{1}_{n_1}\mathds{1}_{n_2}^\top A^\top) + ( (\hat{p} - p)^2 - 2p(p-\hat{p}) )n_2\mathds{1}_{n_1}\mathds{1}_{n_1}^\top )\\
    & = (p-\hat{p})H((W\mathds{1}_{n_2}\mathds{1}_{n_1}^\top+ \mathds{1}_{n_1}\mathds{1}_{n_2}^\top W^\top) \\
    &+ (\delta-1)p(n_{2+}-n_{2-})(\eta_{1}\mathds{1}_{n_1}^\top + \mathds{1}_{n_1}\eta_{1}^\top)   + (p-\hat{p} ) n_2\mathds{1}_{n_1}\mathds{1}_{n_1}^\top ).
\end{align*}
Since
\[
\hat{p} - p = \frac{(\delta-1)p(n_{1+}-n_{1-})(n_{2+}-n_{2-})}{n_1n_2} + \frac{1}{n_1n_2}\sum_{i,j}W_{ij}
\]
then, recalling that $ {|n_{i+}-n_{i-}|}/{n_i}\le \gamma_{i}$ for $i=1,2$, and setting $y: = \frac{1}{n_1n_2}\sum_{i,j}W_{ij}$, we have
$$
|\hat{p} - p| \leq |\delta-1|p\gamma_1\gamma_2 + |y|.
$$
Thus, using again the inequality $\|H(M)\|_\infty \le 2\|M\|_\infty$ and introducing the notation 
$L=\|W\mathds{1}_{n_2}\mathds{1}_{n_1}^\top+ \mathds{1}_{n_1}\mathds{1}_{n_2}^\top W^\top\|_{\infty} $ we obtain
\begin{align*}
\|Z_4\|_{\infty} &\leq 2|p-\hat{p}| L + 4|\delta-1||p-\hat{p}|pn_1n_2 +|p-\hat{p}|^2n_2\|H(\mathds{1}_{n_1}\mathds{1}_{n_1}^\top)\|_{\infty}\\
&\leq 
2|\hat{p}-p|L +  4(\delta-1)^2p^2n_1n_2\gamma_1\gamma_2+ 4 |y|pn_1n_2  +|p-\hat{p}|^2n_1n_2\\
&\leq
 V+ 6(\delta-1)^2p^2n_1n_2\gamma_1\gamma_2,
\end{align*}
where
$$
V =
2|\hat{p}-p|L + 4 |y|pn_1n_2 +2y^2n_1n_2.
$$
Now, note that since $W_{ij}$ are zero mean random variables
\begin{equation*}%\label{eq:hatpp}
\mathbb{E}(y^2) \leq   \frac{2p}{n_1n_2}, \quad \mathbb{E}(|\hat{p} - p|^2)\le p^2 +   \frac{2p}{n_1n_2} .
\end{equation*} 
Moreover, by the same argument as in the control of $\|Z_2\|_{\infty}$,
\[
\mathbb{E}(L^2)=\mathbb{E}(\|W\mathds{1}_{n_2}\mathds{1}_{n_1}^\top+ \mathds{1}_{n_1}\mathds{1}_{n_2}^\top W^\top\|_{\infty}^2) \leq {32}n_2n_1^2p^3\le {8}n_2n_1^2p.
\]
Using these inequalities and the facts that $p\le 1/2$, $n_2\ge 2$ and $\sqrt{p^{2}n_{1}{n_{2}}}\ge C(\delta-1)^{-2}\sqrt{\log{n_1}}\ge 289\sqrt{\log{2}}$ we obtain
\begin{align*}
\mathbb{E}(V) &\leq 2\sqrt{\mathbb{E}(|\hat{p} - p|^2)}\sqrt{\mathbb{E}(L^2)} 
+ 4\sqrt{\mathbb{E}(y^2)}pn_1n_2 +2\mathbb{E}(y^2)n_1n_2\\
&\leq 2 n_1 \sqrt{8n_2p}\sqrt{p^2 +\frac{2p}{n_1n_2}} + 4\sqrt{2n_1n_2p^3} + 4p \\
    &\leq 12\sqrt{p^{2}n_{1}{n_{2}}}(1+\sqrt{pn_1}).
\end{align*}
Putting the above arguments together and applying Markov inequality we get that, for $\gamma_1\gamma_2 \leq \sqrt{\alpha}/96$, 
\begin{align*}
\mathbb{P}\left(\|Z_{4}\|_{\infty} \geq \frac{\sqrt{\alpha}}{12}(\delta-1)^{2}p^{2}n_{1}{n_{2}}\right)  &\leq 
\mathbb{P}\left(V \geq \frac{\sqrt{\alpha}}{48}(\delta-1)^{2}p^{2}n_{1}{n_{2}}\right)\\
&\leq \frac{576 (1+\sqrt{pn_1}) }{(\delta-1)^{2}\sqrt{\alpha} \sqrt{p^{2}n_{1}n_{2}}}
\\&
\leq \frac{576}{(\delta-1)^2\sqrt{\alpha}} ((p^{2}n_{1}{n_{2}})^{-1/2} + (pn_2)^{-1/2}).
%\\
%&\leq \frac{2^{10}}{C \sqrt{\alpha}|\delta-1|}\left( \frac{|n_{1+}-n_{1-}|}{n_1}\frac{|n_{2+}-n_{2-}|}{n_2} +\frac{|\delta-1|^{-1}}{(n_1n_2)^{1/4}} \right).
\end{align*}
Recall that, by the assumptions of the theorem, we have $p \geq C (\delta-1)^{-2}\sqrt{\frac{\log n_{1}}{n_{1}n_{2}}}$, $p \geq C (\delta-1)^{-2}\frac{\log n_{1}}{n_{2}}$, and that we have chosen $C\ge 289/\sqrt{\alpha}$.  Using these inequalities and the facts that $|\delta-1|<1$, $\alpha\in (0,1)$ in the last display we find 
\begin{align*}
\mathbb{P}\left(\|Z_{4}\|_{\infty} \geq \frac{\sqrt{\alpha}}{12}(\delta-1)^{2}p^{2}n_{1}{n_{2}}\right) 
&\leq \frac{36}{\alpha^{1/4} |\delta-1|\sqrt{\log{n_1}}}.
\end{align*}
\end{itemize}
 In conclusion, we have proved that, for any $\alpha\in (0,1)$, $C\ge 289/\sqrt{\alpha}$ and  $\gamma_1\gamma_2 \leq \sqrt{\alpha}/96$ we have
 \begin{align}\label{eq:etta}
   %\underset{n_{1}\to \infty}{\lim} 
   \mathbb{P}\left(\frac{16}{(\delta-1)^{4}p^{4}n_{1}^{2}n_{2}^{2}} \| Z_1+Z_2 +Z_4 \|_{\infty}^2 \geq \alpha\right) & \le 2n_{1}^{-\frac{1}{288}} +
   \frac{9\cdot 2^{9}}{C \alpha  \log n_{1}}\\
   & + \frac{36}{\alpha^{1/4} |\delta-1|\sqrt{\log{n_1}}}.  \nonumber
\end{align}
Hence, if $\alpha\in (0,1)$,  $\gamma_1\gamma_2 \leq \sqrt{\alpha}/96$, there exists an absolute constant $C_0>0$ such that 
 for $C>C_0/\sqrt{\alpha}$ we have
\begin{equation}\label{weakR1}
  \underset{n_{1}\to \infty}{\lim}\mathbb{P}\left(\frac{1}{n_1} r(\eta_1, \eta_1^0) \geq \alpha\right) = 0.
\end{equation}
This proves part (i) of the theorem.
Next, if we assume that $\gamma_1\gamma_2\le 1/C'_{n_1}$ and set $C=C_{n_1}$ where $C_{n_1}, C'_{n_1}$ are any positive sequences that tend to infinity then \eqref{weakR1} holds simultaneously for all $\alpha\in (0,1)$, which proves almost full recovery.

\begin{lem}\label{lem:surrogate} For $\eta_1^0 = \sign (\hat v)$ we have
\begin{align*}
    \frac{1}{n_1} r(\eta_1, \eta_1^0) \leq 2\min_{\nu \in \{-1, 1\}} \left \| \frac{\nu}{\sqrt{n_{1}}}\eta_1 -  \hat{v}  \right\|_2^2.
    \end{align*}
\end{lem}
\begin{proof}
By definition, $r(\eta_1, \eta_1^0)=2\min_{\nu \in \{-1, 1\}}  \sum_{i=1}^{n_1} \mathds{1}\left(\nu\eta_{1i} \neq \eta^0_{1i}\right).$
Set $\hat b = \hat v \sqrt{n_{1}}$. Then $\eta_1^0 = \sign (\hat b)$ and, for any $\nu \in \{-1, 1\}$,
$$
\left \| \frac{\nu}{\sqrt{n_{1}}}\eta_1 -  \hat{v}  \right\|_2^2= \frac{1}{n_1}  \| \nu\eta_1 -  \hat{b} \|_2^2\ge \frac{1}{n_1}\sum_{i=1}^{n_1} \mathds{1}\left(\nu\eta_{1i} \neq \eta^0_{1i}\right),
$$
where the last inequality is due to the fact that $(x-y)^2 \ge \mathds{1}\left(x \neq \sign (y)\right)$ for any $x\in \{-1, 1\}$ and $y\in\mathbb{R}$. 
\end{proof}

\subsection{Proof of {Theorem} \ref{th:th3}}

Note that the assumptions of Theorem \ref{th:th1}(i) are satisfied with $\alpha = 1/25$. Note also that 
$|\eta_1 -  \eta_1^0 |=n_1-\eta_1^\top \eta_1^0$. 
It follows from Theorem \ref{th:th1} and the definition of $r(\hat{\eta}_1, \eta_1)$ that with probability that tends to 1 as $n_1\to \infty$ we have
either $\frac1{n_1}\eta_1^\top \eta_1^0\geq 3/4$ or $\frac1{n_1}\eta_1^\top \eta_1^0\leq - 3/4$. Next, recall that
\begin{align*}
\Gamma :=H((A-\hat{p}\mathds{1}_{n_1}\mathds{1}_{n_2}^\top )(A-\hat{p}\mathds{1}_{n_1}\mathds{1}_{n_2}^\top )^{\top}) = (\delta-1)^2 p^2 n_2 \eta_1 \eta_1^\top + Z.
\end{align*}
From \eqref{eq:etta} we have
$$
\underset{n_{1} \to \infty}{\lim} \mathbb{P}\left(\|Z_1+Z_2+Z_4\|_{\infty} \geq \frac{1}{20}(\delta-1)^{2}p^{2}n_{1}{n_{2}}\right) = 0,
$$
using the same notation as in the proof of Theorem \ref{th:th1}. Observing that $\|Z_3\|_\infty = (\delta-1)^2p^2n_2$ we get moreover that
\begin{equation}\label{eq:zz}
	\underset{n_{1} \to \infty}{\lim} \mathbb{P}\left(\|Z\|_{\infty} \geq \frac{1}{16}(\delta-1)^{2}p^{2}n_{1}{n_{2}}\right) = 0.
\end{equation}
\
%\textcolor{red}{In what follows, we will denote $\eta_1$ (respectively, $\hat{\eta}_1$) by $\eta$ (respectively, by $\hat{\eta}$), since there is no possible confusion between $\eta_1$ and $\eta_2$ in this section.} \\
Define the following random events: 
\begin{align*}
    & O_i = \left\{ \left( \frac{1}{n_1} \Gamma_i \eta_1\right) {\eta_{1i}} \geq \frac{(\delta-1)^2}{2} p^2 n_2  \right\}, \quad i=1,\dots, n_1, \\
    & B = \left \{ \frac{1}{n_{1}}\| Z \|_{\infty} \leq \frac{1}{16}(\delta-1)^{2}p^{2}{n_{2}} \right \},
\end{align*}
where $\Gamma_i$ denotes the $i$th row of matrix $\Gamma$. From \eqref{eq:zz} we have that the probability of $B$ tends to 1 as 
$n_1\to \infty$.
We call $O_{i}$ the oracle events since they are similar to the events arising in the analysis of the supervised oracle procedure that, given the labels $(\eta_{1j}, j\ne i)$, estimates the label $\eta_{1i}$. The proof is decomposed in three steps that we detail in what follows.
\begin{itemize}
    \item \textbf{Proving the contraction.}\\
We place ourselves on the random event $B\cap O_1 \cap \dots \cap O_{n_1}$. Our first goal is to prove that if $\frac1{n_1}\eta_1^{\top}\hat{\eta}^{k} \geq 3/4$, then $|\hat{\eta}^{k+1} - \eta_1| \leq \frac{1}{4} |\hat{\eta}^{k} - \eta_1| 
$  and $\frac1{n_1}\eta_1^{\top}\hat{\eta}^{k+1} \geq 3/4$.  We have
\begin{align*}
    \frac{1}{n_1} \Gamma_i \hat{\eta}^k =& \frac{1}{n_{1}}z_i^\top (\hat{\eta}^k - \eta_1) + \frac{1}{n_1} \Gamma_i \eta_1 \\
    &- (\delta-1)^2p^2 n_2 \eta_{1i} \left(1 - \frac{1}{n_1}\eta_1^\top \hat{\eta}^k \right),
\end{align*}
where $z_i$ is a column vector equal to the transposed $i$th row of matrix $Z$.
Hence, if $\eta_{1i}=-1$ then 

\begin{align*}
    \frac{1}{n_1} \Gamma_i \hat{\eta}^k \leq \frac{1}{n_1}z_i^\top (\hat{\eta}^k - \eta_1) - \frac{(\delta-1)^2}{4} p^2 n_2.
\end{align*}
It follows that
\begin{align*}
    \mathds{1}_{\left\{\frac{1}{n_1} \Gamma_i \hat{\eta}^k \geq 0 \right\}} \leq \mathds{1}_{\left \{ \frac{1}{n_1}z_i^\top (\hat{\eta}^k - \eta_1) \geq \frac{(\delta-1)^2}{4} p^2 n_2 \right \}} \leq 
    \left( \frac{4z_i^\top (\hat{\eta}^k - \eta_1)}{n_1(\delta-1)^2 p^2 n_2} \right)^2.
\end{align*}
Similarly, if $\eta_{1i} = 1$ then
\begin{align*}
    \mathds{1}_{\left\{\frac{1}{n_1} \Gamma_i \hat{\eta}^k \leq 0 \right\}} \leq
     \left(  \frac{4z_i^\top (\hat{\eta}^k - \eta_1)}{n_1(\delta-1)^2 p^2 n_2} \right)^2.
\end{align*}
Now,
\begin{align*}
    \frac{1}{2} \rvert \hat{\eta}^{k+1} - \eta_1 \rvert= \sum_{i = 1}^{n_1} \mathds{1}_{\big\{\frac{1}{n_1} \Gamma_i \hat{\eta}^k \geq 0 \big\}}
    \mathds{1}_{\eta_{1i} = -1} + \sum_{i = 1}^{n_1} \mathds{1}_{\big\{\frac{1}{n_1} \Gamma_i \hat{\eta}^k \leq 0 \big\}}\mathds{1}_{\eta_{1i} = 1}.
\end{align*}
Hence, we get
\begin{align}
\label{eq:eq5}
    \frac{1}{2n_1} \rvert \hat{\eta}^{k+1} - \eta_1 \rvert \leq \left(\frac{4\| Z \|_{\infty}}{n_{1}(\delta-1)^2p^2n_2}\right)^{2} \ \frac{\| \hat{\eta}^k - \eta_1 \|_2^2}{n_1} \leq \frac{1}{8n_1} \rvert \hat{\eta}^{k} - \eta \rvert.
\end{align}
The fact that $\frac1{n_1}\eta_1^{\top}\hat{\eta}^{k+1} \geq 3/4$ follows immediately from the inequality $|\hat{\eta}^{k+1} - \eta_1| \leq \frac{1}{4} |\hat{\eta}^{k} - \eta_1| 
$ and the relation
$\rvert \hat{\eta}^{k} - \eta_1 \rvert=n_1-\eta_1^{\top} \hat{\eta}^{k}$.

Quite analogously, we find that that if $\frac1{n_1}\eta_1^{\top}\hat{\eta}^{k} \leq - 3/4$, then $|\hat{\eta}^{k+1} + \eta_1| \leq \frac{1}{4} |\hat{\eta}^{k} + \eta_1|.$ 

\item \textbf{Reduction to the oracle events.}\\
   Assume that the event $B\cap O_1 \cap \dots \cap O_{n_1}$ holds. Let first $\frac{1}{n_1}\eta_1^{\top} {\eta}^{0}_1 \geq 3/4$. Since $\rvert {\eta}^{0}_1 - \eta_1 \rvert=n_1-\eta_1^{\top} {\eta}^{0}_1$  we get
\begin{align*}
    \frac{1}{n_1} \rvert \hat{\eta}^k - \eta_1 \rvert \leq
    \frac{1}{n_1} \rvert {\eta}^{0}_1 - \eta_1 \rvert \left( \frac{1}{4} \right)^k \le \left( \frac{1}{4} \right)^{k+1}.
\end{align*}
For $k > \frac{\log n_1}{2\log 2}-\frac{3}{2}$ we have 
\begin{align*}
    \left( \frac{1}{4}\right)^{k+1} < \frac{2}{n_1},
\end{align*}
so that 
\begin{align*}
     \rvert \hat{\eta}^k - \eta_1 \rvert = 0.
\end{align*}
Quite similarly we prove that if $\frac{1}{n_1}\eta_1^{\top} \hat{\eta}^{k} \leq - 3/4$ then, 
for $k > \frac{\log n_1}{2\log 2}-\frac{3}{2}$,
\begin{align*}
     \rvert \hat{\eta}^k + \eta_1 \rvert = 0.
\end{align*}
Recalling the definition of $r(\hat{\eta}^k, \eta_1)$ we conclude that 
$$
\mathbb{P}\left( r(\hat{\eta}^k, \eta_1) \neq 0 \right) \leq \mathbb{P}(B^{c}) + \sum_{i=1}^{n_1}\mathbb{P}(O_{i}^c).
$$
It follows from \eqref{eq:zz} that $\underset{n_1 \to \infty}{\lim}\mathbb{P}(B^{c}) = 0$. Thus, the proof of the theorem will be complete if we show that 
\begin{equation}\label{eq:qqq}
	 \underset{n_{1} \to \infty}{\lim}\sum_{i=1}^{n_1}\mathbb{P}(O_{i}^c) = 0.
\end{equation}
 
 \item \textbf{Control of the oracle events.}\\
 We proceed now to the proof of \eqref{eq:qqq}.
 \noindent Let $G_1, \dots, G_{n_1}$ be the column vectors equal to the transposed rows of matrix $G:=A-\hat{p} \mathds{1}_{n_1}\mathds{1}_{n_2}^{\top}= (p-\hat{p}) \mathds{1}_{n_1}\mathds{1}_{n_2}^{\top}+ (\delta-1)p \eta_{1}\eta_{2}^{\top} + W$.  For all $i=1,\dots,n_1$, we have
\begin{align*}
\mathbb{P}(O_{i}^c)=\mathbb{P} \left(\eta_{1i} G_{i}^{\top}\left(\sum_{k \neq i} \eta_{1k} G_{k}\right) < \frac{(\delta-1)^2}{2}p^2n_2n_1 \right). 
\end{align*}
Denoting by  $X_{1},\dots,X_{n_1}$ the column vectors equal to the transposed rows of matrix $W$ we may write $\eta_{1i}G_i = v_i + \eta_{1i}X_i$, where $v_i = \eta_{1i}(p-\hat{p})\mathds{1}_{n_2}+ (\delta-1)p\eta_2 $. Therefore, 
\begin{align*}
\eta_{1i} G_{i}^{\top}\left(\sum_{k \neq i} \eta_{1k} G_{k}\right)& = 
(v_i^\top + \eta_{1i}X_i^{\top})\left(\sum_{k \neq i} v_k + \sum_{k \neq i}\eta_{1k}X_k\right)
\\
& = (\delta-1)^2p^2n_2(n_1-1) + T_1 + T_2 +T_3 + T_4,
\end{align*}
where
$$
T_1 =  \eta_{1i}  \sum_{k \neq i}X_i^\top v_k,
\quad T_2=   \sum_{k \neq i}\eta_{1k}v_i^\top X_k,
$$
$$
T_3=
\eta_{1i}\sum_{k \neq i}\eta_{1k}X_i^{\top}X_k,
\quad T_4 = \sum_{k \neq i} v_i^\top v_k - (\delta-1)^2p^2n_2(n_1-1)
$$
and we obtain
$$
\mathbb{P}(O_{i}^c)= \mathbb{P}\left(-T_1 - T_2 -T_3 - T_4 > (\delta-1)^2p^2n_2(n_1/2 - 1)\right).
$$
We now bound from above the four corresponding probabilities. 
First, recall that
\[
|\hat{p} - p| \le |\delta-1|p\gamma_1\gamma_2+ \Big|\frac{1}{n_1n_2}\sum_{i,j}W_{ij}\Big|\le \frac{|\delta-1|p}{480} + \Big|\frac{1}{n_1n_2}\sum_{i,j}W_{ij}\Big|. 
\]
The entries $W_{ij}$ of matrix $W$ are independent zero-mean random variables distributed as  $\zeta - \mathbb{E}(\zeta)$ where $\zeta$ is a Bernoulli random variable with parameter $\delta p$ or $(2-\delta)p$. As $W_{ij}$ are bounded in absolute value by 1 and have variances bounded by $2p$ we get from Bernstein's inequality that 
\begin{equation}\label{last-1}
  \mathbb{P}(|\hat{p}-p|\geq |\delta-1|p/64) \leq 2 e^{-c(\delta-1)^2n_1n_2p}.  
\end{equation}
%where the last inequality uses the assumption that $p \geq C(\delta-1)^{-2}\sqrt{\frac{\log n_1}{n_1 n_2}}$.
Here and below we denote by $c$ absolute positive constants that may vary from line to line. 
Next, on the event $|\hat{p}-p|\leq |\delta-1|p/64$ we have
\begin{align*}
|T_1|&\le |\mathds{1}_{n_2}^\top X_i|\Big|\sum_{k \neq i}\eta_{1k}(p-\hat{p})\Big| + |\delta-1|(n_1-1)p |\eta_{2}^\top X_i|\\
&\le |\delta-1|pn_1(|\mathds{1}_{n_2}^\top X_i| +  |\eta_{2}^\top X_i|).
\end{align*}
Here, $\mathds{1}_{n_2}^\top X_i$ and $\eta_{2}^\top X_i$ are two sums of $n_2$ independent zero-mean random variables  bounded in absolute value by 1 and with variances bounded by $2p$.
Using these remarks, Bernstein's inequality and \eqref{last-1} we obtain that, for $n_1\ge 4$,
\begin{align*}
&\mathbb{P}\left( |T_1| \geq \frac{1}{4}(\delta-1)^2p^2n_2(n_1/2 - 1) \right)  \\
&\leq \mathbb{P}\left(
|\mathds{1}_{n_2}^\top X_i| +  |\eta_{2}^\top X_i|\geq \frac{1}{16}|\delta-1|p n_2 \right)   + \mathbb{P}(|\hat{p}-p| \geq |\delta-1|p/64)\\
& \leq 4\exp \left( -c(\delta-1)^2pn_{2}\right) + \mathbb{P}(|\hat{p}-p| \geq |\delta-1|p/64) \\ &\leq 6\exp \left(-c \, C\log n_1 
\right)
 \leq \frac{1}{n_1^{2}}
\end{align*}
where we have used the assumption that %$p \geq C(\delta-1)^{-2}\sqrt{\frac{\log n_1}{n_1 n_2}}$ and 
$ p \geq C(\delta-1)^{-2}\frac{\log n_1}{ n_2}$ for some $C>0$ large enough. Quite analogous application of Bernstein's inequality, this time to two sums of $n_2(n_1-1)$ random variables, yields the bound
\begin{align*}
& \mathbb{P}\left( - T_2 \geq \frac{1}{4}(\delta-1)^2p^2n_2(n_1/2 - 1) \right)
\\
& \qquad \leq \mathbb{P}\left( - \sum_{k \neq i} \eta_{1k}v_i^\top X_k \geq \frac{1}{16}(\delta-1)^2p^2n_1n_2 \right)\\
&\qquad \leq 6 \exp\left(- c (\delta-1)^{2}p n_1 n_{2} \right) 
\\ &\qquad \leq 6 \exp \left(- c \,C n_1\log n_1 \right)
 \leq \frac{1}{n_1^{2}}.
\end{align*}
Next, we consider the term $T_3=
\eta_{1i}\sum_{k \neq i}\eta_{1k}X_i^{\top}X_k$. We have
\begin{align*}
    &\mathbb{P}\left( -T_3\geq \frac{1}{4}(\delta-1)^2p^2n_2n_1 \right) \\ &\qquad\leq 
    \mathbb{E}\left[\mathbb{P}\left( - T_3\geq \frac{1}{4}(\delta-1)^2p^2n_2n_1\Big\vert X_i \right)\mathds{1}_{F_i}\right] + \mathbb{P}(F_i^c),
\end{align*}
where $F_i= \{ \|X_{i}\|_{2}^{2} \leq 6n_{2}p\}$. Recall that $\|X_{i}\|_{2}^{2} = \sum_{j=1}^{n_2}W_{ij}^2$ where $W_{ij}$ are the elements of matrix $W$.
We now apply Bernstein's inequality conditionally 
on $X_i$ to the random variable $T_3$, which is (conditionally 
on $X_i$) a sum of $n_2(n_1-1)$ independent zero-mean random variables 
bounded in absolute value by 1 and with the sum of variances bounded by $2p(n_1-1)\|X_{i}\|_{2}^{2}$. It follows from Bernstein's inequality that  
for any fixed $X_i\in F_i$ we have
\begin{align*}
&\mathbb{P}\left( -\eta_{1i}\sum_{k \neq i}\eta_{1k}X_i^{\top}X_k\geq 
\frac{1}{4}(\delta-1)^2p^2n_2n_1 \Big\vert X_i\right) 
\\
&\qquad\leq 
\exp \left(-\frac{c(\delta-1)^4p^4n_2^2n_1^2}{p n_1\|X_{i}\|_{2}^{2}+(\delta-1)^2p^2n_2n_1}\right) 
\\
& \qquad
\le \exp \left(- c(\delta-1)^4p^2n_2n_1\right)
\leq \frac{1}{n_{1}^2},
\end{align*}
where the last inequality is valid if $C>0$ is large enough. Applying once more Bernstein's inequality
we obtain the bound 
$$ \mathbb{P}(F_i^c) \leq  \mathbb{P}\left( \sum_{j=1}^{n_2}(W_{ij}^2 - \mathbb{E}(W_{ij}^2 ) )\geq 4n_2p \right) \leq \exp \left( -c n_{2}p\right) \leq \frac{1}{n_{1}^2}
$$
 if $C>0$ is large enough. 
Finally, we consider the term $T_4 = \sum_{k \neq i} v_i^\top v_k - (\delta-1)^2p^2n_2(n_1-1)$. We have
\begin{align*}
|T_4|&\le \Big|\eta_{1i}(p-\hat{p})^2n_2\sum_{k \neq i}\eta_{1k}\Big| +
\Big|(\delta-1)p(p-\hat{p})(\eta_{2}^\top\mathds{1}_{n_2})\sum_{k \neq i}\eta_{1k}\Big|\\
& \quad +  
\Big|\eta_{1i}(\delta-1)p(p-\hat{p})(n_1-1)(\eta_{2}^\top\mathds{1}_{n_2})\Big|
\\
&\le n_1n_2 (p-\hat{p})^2 + 2|\delta-1|pn_1n_2 |\hat{p}-p|.
\end{align*}
Therefore, on the event $ |\hat{p}-p|\leq p|\delta-1|/64$ we have $|T_4|<\frac{1}{16}(\delta-1)^2p^2n_1n_2$, which implies that for $n_1\ge 4$ and $C>0$ large enough,
\begin{align*}
   &\mathbb{P}\left(- T_4 \geq \frac{1}{4} (\delta-1)^2p^2n_2(n_1/2 - 1)\right)
   \\
   &\qquad\leq \mathbb{P}( |\hat{p}-p|\geq p|\delta-1|/64)
   %\\
   %&\leq \frac{1}{n_1^2}.
%\end{align*}
\leq 2\exp \left(-c \, Cn_1\log n_1 
\right)
 \leq \frac{1}{n_1^{2}},
\end{align*}
where we have used \eqref{last-1} and the assumption that 
$ p \geq C(\delta-1)^{-2}\frac{\log n_1}{ n_2}$ for some $C>0$ large enough.
Combining the above inequalities we find that, for $C>0$ large enough, 
$$
\sum_{i=1}^{n_1} \mathbb{P}(O_{i}^c) \leq \frac{4}{n_{1}} \underset{n_1 \to \infty}{\to} 0.
$$
This proves \eqref{eq:qqq} and hence the theorem.
\end{itemize}

\end{document}